\documentclass[10pt]{article}
\usepackage{lineno,hyperref}
\usepackage{amsmath}
\usepackage{amsfonts}
\usepackage{amssymb}
\usepackage{latexsym}
\usepackage{indentfirst}
\usepackage{graphicx,geometry}
\usepackage[dvips]{epsfig}
\usepackage{subfigure}
\usepackage{amsthm}
\usepackage{color}
\usepackage{soul}
\usepackage{graphicx}
\usepackage{tikz}%
\usepackage{todonotes} 
\providecommand{\U}[1]{\protect\rule{.1in}{.1in}}
\newtheorem{theorem}{Theorem}
\newtheorem{proposition}[theorem]{Proposition}
\theoremstyle{definition}
\newtheorem{definition}[theorem]{Definition}
\newtheorem{remark}[theorem]{Remark}
\def\VIZ#1{(\ref{#1})}
\def\RENYI#1#2#3{\mathcal{R}_{#1}\left({#2}\SEP{#3}\right)}

\def\SEP{{\,||\,}}

\def\a{\alpha}
\def\b{\beta}
\def\g{\gamma}
\def\d{\delta}
\def\e{\epsilon}
\def\8{\theta}

\newcommand\encircle[1]{\tikz[baseline=(X.base)] \node (X) [draw, shape=circle, inner sep=0] {\strut #1};}
\begin{document}

\title{Sensitivity Analysis for Rare Events based on R\'enyi Divergence}

\author{
Paul Dupuis\thanks{dupuis@dam.brown.edu, Research supported in part by the Defense Advanced Research Projects Agency (DARPA) EQUiPS program (W911NF-15-2-0122), the Air Force Office of Scientific Research (AFOSR) (FA-9550-18-1-0214) and the National Science Foundation (DMS-1317199)}\\ 
{\small Division of Applied Mathematics, Brown University, Providence, RI 02912, USA} 
\and Markos A. Katsoulakis\thanks{markos@math.umass.edu, Research supported in part by the Defense Advanced Research Projects Agency (DARPA) EQUiPS program (W911NF-15-2-0122), the Air Force Office of Scientific Research (AFOSR) (FA-9550-18-1-0214) and the National Science Foundation (DMS-1515712)} 
\\ {\small Department of Mathematics and Statistics, 
University of Massachusetts Amherst, Amherst, MA 01002, USA}
\and  Yannis Pantazis\thanks{pantazis@iacm.forth.gr}
\\{\small  Institute of Applied and Computational Mathematics, Foundation for Research and Technology - Hellas, Heraklion, 70013, Greece} 
\and  Luc Rey-Bellet\thanks{luc@math.umass.edu, Research supported in part by the National Science Foundation (DMS-1515712) and the Air Force Office of Scientific Research (AFOSR) (FA-9550-18-1-0214)}\\ {\small Department of Mathematics and Statistics, University of Massachusetts Amherst, Amherst, MA 01002, USA} 
}
 
\date{\today}

\maketitle
\begin{abstract}
Rare events play a key role in many applications and numerous algorithms have been  proposed for estimating 
the probability of a rare event.  However, relatively little is known on  how to quantify the sensitivity of the
probability with respect to model parameters.  
In this paper, instead of the  direct statistical estimation of rare event sensitivities, we develop novel and general uncertainty quantification and sensitivity bounds which are not tied to specific rare event simulation methods and which apply to families of rare events.  Our method is based on a recently derived variational  representation for the family of R\'enyi divergences in terms of risk sensitive functionals associated with the rare events under consideration. Based on the derived bounds, we propose new sensitivity indices for rare events and relate them to the moment generating function of the score function.  The bounds scale in such a way that we additionally develop sensitivity indices  for  large deviation rate functions.
\end{abstract}


\section{Introduction and Main Result}\label{sec:intro}

Rare events play an important role in a wide range of applications. For example in insurance, finance and 
risk management, rare events play an outsized role due to potentially  catastrophic consequences \cite{pha}.  
In queueing theory the probability of a buffer overflow often needs to be estimated \cite{Rubino:book}, and
in molecular  dynamics, metastability effects play a crucial role in determining the behavior of the 
system \cite{caikaldekbul}.  Similarly, extreme value theory studies  events and statistical samples which are far from the typical observed \cite{Embrechts:1997}. There is a large body of literature on rare event simulation and many different techniques have been developed to approximate the probability of a rare event. For example, importance sampling \cite{Glynn:book,RK:17,dupwan5} transforms the distribution of random variables in order to make the rare event typical and corrects for bias using the likelihood ratio;  interacting particle systems methods \cite{DM:04,DM:05}
use many (dependent) copies of the system to speed the exploration of state space; splitting techniques \cite{Lec:09,altalt1,gar2,deadup2} decompose the problem of a single rare event into a sequence of not so rare events; 
and so on...
Closely related are multi-level methods inspired primarily by statistical mechanics considerations, e.g.  \cite{Liu:MC}, as well as by rare events involving barrier crossing in molecular simulation, e.g. \cite{Frenkel:02, Parrinello:Rev:16}.

In this paper we are primarily interested  in the problem 
of uncertainty quantification (UQ) and in particular sensitivity analysis 
for rare events, a problem of practical 
importance whenever there is uncertainty in 
the parameters of the model or even in the model 
itself.  This problem has rarely been
addressed in the literature despite the fact that the statistics of rare events are often heavily influenced by the particular values of model parameters.
The case of Poisson processes in the context of importance sampling for risk models was considered in \cite{AR:99, Abad:01}, while more recent work, \cite{Gobet:2015},
proposed importance sampling combined with splitting techniques in the context of Gaussian models and 
Malliavin calculus.  Some examples using the likelihood ratio method are also discussed  in \cite{DelMoral:2012}.

Our approach in this paper is not based on any 
specific algorithm for rare event simulation but 
rather on novel information theoretic bounds. 
These bounds  allow us to define a new sensitivity index that is independent of the particular event.
Instead, the bounds hold for all rare events with probability above any  fixed threshold.
Specifically we utilize the R\'enyi family of relative entropies (a.k.a. R\'enyi divergence)
as a measure of uncertainty between probability distributions
and a new variational representation for risk-sensitive functionals in terms of R\'enyi 
relative entropy, derived by Atar et al. \cite{Atar:15}, to obtain general bounds on the uncertainty quantification and sensitivity analysis
of families of rare events.  

We first introduce the main objects of interest, namely sensitivity indices for probabilities of rare events,  and discuss  the challenges involved in their estimation.
We then present the main result of the paper: bounds on the sensitivity indices for the families of 
events whose probability is at least $e^{-M}$, where $M$ is any fixed rare event threshold.
At this stage we assume the degree of rarity is characterized by $M$ but later on a large deviation parameter will be introduced in Section \ref{sec:Sens:LDP}.

\medskip
\noindent
{\bf Gradient sensitivity indices for rare events.}
Let $P^{\theta}$ be a family of probability measures parameterized by a vector
$\theta\in\mathbb{R}^{K}$. We assume that $P^{\theta}\ll R$ where $R$ is a reference measure
and we denote by $p^{\theta}=\frac{dP^{\theta}}{dR}$ the corresponding family of 
densities. We also assume that the mapping $\theta\mapsto p^{\theta}(x)$ satisfies suitable
differentiability and integrability conditions in order to interchange integration and differentiation.
For a rare  event $A$ with $0< P^\theta(A) \ll 1$ we define the sensitivity index 
in the direction $v\in\mathbb{R}^{K}$ by 
\begin{equation}
\label{SI1}
S_{v}^{\theta}(A) = v^{T}\nabla_{\theta}\log P^{\theta}(A)  =   \frac{v^{T} \nabla_{\theta}P^{\theta}(A)}{P^{\theta}(A)} ,
\end{equation}
which describes the relative change of the quantity of interest
$P^{\theta}(A)$ with respect to the parameter $\theta$ in the $v$ direction.

One of the simplest ways to estimate the sensitivity index \eqref{SI1} is by
considering finite difference approximations for each partial derivative, i.e.,
considering all coordinate unit directions $v=e_{i} \in\mathbb{R}^{K}, i=1,
2,..., K$:
\begin{equation}
S_{e_i}^{\theta}(A)  \approx\frac{\log P^{\theta+\epsilon e_{i}}(A)-\log
P^{\theta}(A)}{\epsilon} \, .\label{SI3}%
\end{equation}
However, the cost of implementing such an approximation can be prohibitive, given the cost
of estimating the small probabilities $P^{\theta+\epsilon e_{i}}(A),
i=1,2,...,K$. A  variant of the likelihood ratio method can at least partially address this
issue, as we discuss next.

\medskip
\noindent
{\bf Likelihood ratio method for rare events.}
The gradient sensitivity for expected values of observables 
can be computed using various methods, such as the likelihood ratio method and infinitesimal perturbation analysis, see e.g. \cite{Glynn:book}.  While these methods are in principle applicable to the problem of computing the rare event sensitivity index \eqref{SI1} (as we show next in the context of likelihood ratio), they still  
require  the  use of some form of accelerated Monte Carlo simulation, for example importance sampling \cite{Glynn:book}, to (possibly) obtain acceptable performance when rare events are involved.

We define the score function for the parametric family $P^\theta$ by 
\begin{equation}
W^{\theta}(x):=\nabla_{\theta}\log p^{\theta}(x)\,,\label{score}%
\end{equation}
with the convention that $W^{\theta}(x)=0$ if $p^\theta(x)=0$.  
We also denote by $P^{\theta}_{|A}$ the probability $P^\theta$ conditioned on the event $A$, i.e.,
$P^{\theta}_{|A}(dx)=\frac{1}{P^{\theta}(A)}\chi_{A}(x)p^{\theta}(x)R(dx)$, where $\chi_{A}$ is the indicator function.

Under suitable conditions to ensure the interchangeability between integrals and  derivatives, 
the sensitivity index for the rare event $A$ given in \eqref{SI1} can be rewritten as 
\begin{equation} \label{LRR} 
S_{v}^{\theta}(A) = \frac{  v^T \nabla_\theta \int_A  p^\theta  dR}{P^{\theta}(A)} \,=\, 
\frac{  v^T \int_A  W^\theta  dP^\theta}{P^{\theta}(A)} =  
  \mathbb{E}_{P^{\theta}_{| A}} \left[ v^T W^{\theta} \right]  \,. 
\end{equation}
An algorithm that estimates \eqref{LRR} by combining the likelihood ratio method with importance sampling through interacting particles was recently developed in \cite{DelMoral:2012}.
 Both approaches, \eqref{SI3} and \eqref{LRR},
 are  feasible only when an accelerated Monte Carlo scheme appropriate to the particular event $A$ has been designed.
Therefore, sensitivity analysis methods for rare events  which would apply to a  whole class of events $A$ (or more generally expected values which are sensitive to rare events), would be a more  practical computational tool.
One approach to this end is to derive  upper and lower bounds for $S_{v}^{\theta}(A)$ that can serve as  new sensitivity indices. 
These are of course less accurate, but may be much easier to compute, and can be used to identify those parameters for which greater accuracy is not needed. We show next that the well-known Cramer-Rao type bounds are not useful for the sensitivity analysis of rare events.

\medskip 
\noindent
{\bf Failure of Cramer-Rao type bounds.}  The sensitivity index for a regular, i.e., non-rare, observable has the form  
$\nabla_\theta \mathbb{E}_{P^\theta}[f]=\mathbb{E}_{P^\theta}[ f W^\theta]$. This can be easily  bounded using the Cramer-Rao inequality, \cite{Kay:1993}
i.e., 
\[
\left| v^T \nabla_\theta \mathbb{E}_{P^\theta}[f] \right| \,\le \,  \sqrt{ {\rm Var}_{P^\theta} [f] } \sqrt{ v^T  \mathcal{F}(P^{\theta})
v}
\]
where $\mathcal{F}(P^{\theta}) = \mathbb{E}_{P^{\theta}}[ W^\theta (W^\theta)^T]$ is the Fisher information matrix.  Applying the Cramer-Rao bound to a rare event $A$ yields 
\begin{equation}
\label{eq:cramer-rao}
| S_v^\8(A) | = \frac{1}{P^\8(A)} \left| \mathbb E_{P^\theta} [\chi_A v^T W^\theta] \right| 
\le  \sqrt{\frac{1-P^\8(A)}{P^\8(A)}}  \sqrt{v^T \mathcal{F }(P^\8) v} \, .
\end{equation}
Unfortunately, this (naive) sensitivity bound is rather useless since it scales as $P^{\theta}(A)^{-1/2}$. This can be very large for a rare event $A$, while one expects the sensitivity index to be of order $O(1)$.

\medskip
\noindent
{\bf Information-based sensitivity indices for rare events.}
In view of the difficulty of directly approximating \eqref{SI1},
the main contribution in this 
paper is to develop information-based bounds for the 
sensitivity indices defined in \eqref{SI1} that apply to {\it families} of rare  events.
The bounds  involve only a single risk-sensitive  functional for the score function $W^\theta$.
While this quantity must be approximated, the bound does not require a different rare 
event sampler for each distinct rare event. 
One of the main results proved in this paper is presented next, while complete technical assumptions on $P^\theta$ 
are given in Section \ref{sec:Sens:bounds}.



\smallskip
\noindent
{\bf \em Main result: Sensitivity bounds and a sensitivity index for rare events.}
{\em
For $0 < M < \infty$ let \begin{equation*}
\bar{\mathcal{A}}_{M}=\{ A \,:\,   P^{\theta}(A) = e^{-M} \}\,.
\label{familyR}
\end{equation*}
denote the sets, parametrized by the positive constant $M$, of all events which are equally probable (or rather equally rare if $M\gg 0$). 
Then, for any $A \in \bar{\mathcal{A}}_M$ we have  
\begin{equation}
\label{eq:inequality:index:intro}
\mathcal{I}_{v,-}^{\theta}(M) \le S_{v}^{\theta}(A)
\le \mathcal{I}_{v,+}^{\theta}(M)\,,  
\end{equation}
where 
\begin{equation}
\mathcal{I}_{v,\pm}^{\theta}(M):= \pm  \inf_{\alpha>0}\left\{
\frac{H_{v}^{\theta}(\pm \alpha)+M}{\alpha}\right\}  \quad {\rm with} \quad 
H_{v}^{\theta}(\alpha)=\log\mathbb{E}_{P^{\theta}}\left[  e^{\alpha
v^{T}W^{\theta}}\right] \,. 
\label{InfoSI}%
\end{equation}
(Note that $H_{v}^{\theta}(\alpha)$  is the cumulant generating function for the score function 
$W^{\theta}$ defined in \eqref{score}.)
Furthermore, denoting by  $P^\theta_\alpha$ the exponential family of tilted measures \[
\frac{dP_\alpha^\theta}{dP^\theta} = e^{  \alpha v^T W^\theta - H_v^\theta(\alpha)}\,,
\]  
we have 
\begin{equation}\label{InfoSI2}
\mathcal{I}_{v,\pm}^{\theta}(M) = \frac{d}{d\alpha} H_v^\theta(\alpha) \Big|_{\alpha=\pm\alpha_\pm} = \mathbb{E}_{P^{\theta}_{\alpha_\pm}}  \left[ v^{T}W^{\theta} \right]\,.
\end{equation}
Here  $\alpha_\pm$ are determined by 
 \[
\mathcal{R} ( { P^{\theta}_{\pm \alpha_{\pm}}}{\,||\,}{P^{\theta}}) =  M \,,
\]
and $ \mathcal{R} ( Q {\,||\,} P)$ denotes the relative entropy of $Q$ with respect to $P$.  See Figure~\ref{prop:3:1:fig} for a graphical depiction of the characterization of $\alpha_{\pm}$.
}
\medskip


The proposed rare event sensitivity indices \eqref{InfoSI} are  bounds  
for the gradient-based indices \eqref{SI1}. They do  not require a rare 
event sampler for {\em each} rare  event $A$, 
as one readily sees in the definition of \eqref{InfoSI} or \eqref{InfoSI2}, 
and they apply to the entire class $\bar{\mathcal{A}}_{M}$  of rare events, 
that is for the  probability level sets for the parametric  model 
$P^\theta$.
Intuitively, $\mathcal{I}_{v,\pm}^{\theta}(M)$ balances between the rarity of the event as quantified by $M$
and the cost to be paid in order to make the event less rare as quantified by $H_v^\theta(\alpha)$.
Note also that,  due to the monotonicity of \eqref{InfoSI} in $M$, 
the rare event sensitivity indices $\mathcal{I}_{v}^{\theta, \pm}(M)$ 
actually characterize the sensitivity of the model $P^\theta$ for each 
family 
$$
{\mathcal{A}}_{M}:=\{ A:  P^{\theta}(A)\ge e^{-M}\}\, ,
$$
i.e., all events which are less rare than the threshold  $e^{-M}$.
In this sense, the bounds \eqref{eq:inequality:index:intro} 
present similar computational advantages and  trade-offs as 
other sensitivity bounds  for  typical observables (not rare 
event dependent), such as the Cram\'er-Rao information  
bounds, see \eqref{eq:cramer-rao}.  Namely they are less 
accurate than the exact gradient-based indices \eqref{SI1},
but they can be  used  to determine insensitive parameters 
and directions $v$, without requiring  recalculation for 
different events $A$.  We present a simple demonstration of such an insensitivity analysis based on 
the bounds \eqref{eq:inequality:index:intro}  in the last example of Section~\ref{sec:examples}.

Additionally, in ordinary Cramer-Rao bounds, the 
sensitivity of the parametric model is encoded into the 
Fisher information matrix (the variance of the score 
function), for rare event the cumulant generating function of 
the score function plays a central role. Since the cumulant 
generating function controls rare events (as in Cramer's 
Theorem) our bounds show that the rare events associated to 
the score function control the sensitivity of all rare events.
The question of the tightness of the sensitivity bounds will 
be  addressed in \cite{DKRW:17}, which discusses the UQ and sensitivity 
analysis for more general risk-sensitive functionals. 

For a  practical implementation of the bound one can use concentration equalities \cite{Boucheron:2016} in a similar manner  to UQ bounds for regular observables,  \cite{Gourgoulias:17}, and obtain easily computable but less accurate bounds,
see Section \ref{sec:estimation}. 
As the bound involves the moment generating function of the score function, the  rare event simulation techniques mentioned at the beginning of the introduction could also in principle be used to solve the optimization problem in the sensitivity indices.  We plan to revisit this issue in a follow-up publication.

The paper is organized as follows. In Section 
\ref{sec:math:prelim} we begin with the study of an
optimization problem which appears several times throughout the 
paper and then proceed with the definition and the properties of
the R\'enyi divergence. In Section \ref{sec:UQ:bounds} 
we derive our main UQ bounds
based on the R\'enyi divergence optimized over its 
parameter. We then derive in Section 
\ref{sec:Sens:bounds} information inequalities for the sensitivity 
indices.  
In Section~\ref{sec:estimation} we discuss the practical implementation of the sensitivity indices for rare events, 
via concentration inequalities or  via  numerical estimation.
%
We illustrate our results on several distributions from the exponential 
family in Section \ref{sec:examples},  and in Section \ref{sec:Sens:LDP}, we present 
sensitivity bounds for large deviation rate 
functions. 

\section{Mathematical Preliminaries}
\label{sec:math:prelim}

\subsection{An optimization problem}

\label{sec:opt:problem}

In order to obtain optimal UQ bounds we have to consider a certain
optimization problem involving cumulant generating functions. We note that a
bound function with similar structure has been derived and studied recently in
\cite{Chowdhary:13, Li:12, Dupuis:16}, and we slightly generalize and reformulate those
results in this section.

Let $\mathcal{X}$ be a Polish space, $\mathcal{B}(\mathcal{X})$ the associated
Borel $\sigma$-algebra and denote by $\mathcal{P}(\mathcal{X})$ the set of all
probability measures on $(\mathcal{X},\mathcal{B}(\mathcal{X}))$.
Given a probability measure $P \in\mathcal{P}(\mathcal{X})$ and a measurable
function $g: \mathcal{X} \to\mathbb{R}$ consider the moment
generating function $\mathbb{E}_{P}\left[  e^{\alpha g}\right] $ with
$\alpha\in\mathbb{R}$. We will assume that $g$ is such that the moment
generating function is finite in a neighborhood of the origin and denote the
space of such functions by $\mathcal{E}$.
%

\begin{definition}\label{def:gdef}
A measurable function $g:\mathcal{X}\to\mathbb{R}$ belongs to the set
$\mathcal{E}$ if and only if there exists $\alpha_{0}>0$ such that
$\mathbb{E}_{P}\left[  e^{\pm\alpha_{0} g}\right]  < \infty$\,.
\end{definition}

If $g\in\mathcal{E}$ then as is well known $g$ has finite
moments of all orders, see also the discussion Appendix \ref{app:a}.

\begin{definition}
\label{CGF:LT:gen:def} Given $P \in\mathcal{P}(\mathcal{X})$ and $g
\in\mathcal{E}$ the \textit{cumulant generating function} of $g$ is given by
\begin{equation*}
H(\alpha) := \log\mathbb{E}_{P} [e^{\alpha g}] \ . \label{CGF:gen}%
\end{equation*}
A family of probability measures naturally associated to $H(\alpha)$ is the
exponential family $P_{\alpha}$ given by
\begin{equation*}
\frac{dP_{\alpha}}{dP} \,=\, e^{ \alpha g - H(\alpha)} \,,
\end{equation*}
which is well-defined if $H(\alpha) < \infty$.
\end{definition}
In Appendix~\ref{app:a}, we summarize various useful properties of cumulant generating functions. These will be needed to study the following minimization problems, which arise in the definition of the sensitivity indices introduced in Section~\ref{sec:intro}.

%
%

\begin{proposition}
 \label{aux:bound:func:gen:prop} Let $P \in\mathcal{P}(\mathcal{X})$ and $g
\in\mathcal{E}$ with $g$ not a constant $P$-a.s. Suppose $(d_-,d_+)$ is the largest open set such that $H(\alpha)<\infty$ for all $\alpha \in (d_-,d_+)$.

\begin{enumerate}
\item For any $M \ge0$ the optimization problems
\begin{equation*}
\inf_{\alpha>0} \frac{H(\pm\alpha)+M}{\alpha}
\end{equation*}
have unique minimizers $\alpha^{\pm}\in[0, \pm d_\pm]$.   Let $M_\pm$ be defined by 
\[
M_\pm = \lim_{\alpha \nearrow \pm d_\pm}   \pm \alpha H'( \pm \alpha) - H(\pm \alpha) \,.
\]
Then 
the minimizers $\alpha_{\pm}=\alpha_{\pm}(M)$
are finite for $M < M_{\pm}$ and $\alpha_{\pm}(M) = \pm d_\pm$ if $M \geq
M_{\pm}$.

\item If $\alpha_{\pm}(M) < \pm d_{\pm}$ then
\begin{equation}
\frac{H(\pm\alpha_{\pm})+M}{\alpha_{\pm}}=
\inf_{\alpha>0} \frac{H(\pm\alpha)+M}{\alpha}
\,=\,\pm H^{\prime}%
(\pm\alpha_{\pm})=\pm \mathbb{E}_{P_{\pm\alpha_{\pm}}}[g]\,,\label{eq:Bpmeqn}%
\end{equation}
where $\alpha_{\pm}(M)$ is strictly increasing in $M$ and is determined by the
equation
\begin{equation}
\mathcal{R}\left(  {P_{\pm\alpha_{\pm}}}{\,||\,}{P}\right)
=M\,.\label{eq:cond:re}%
\end{equation}

\item $M_{\pm}$ is finite in two distinct cases.
\begin{enumerate}
\item If $\pm d_\pm < \infty$ (in which case $g$ is unbounded above/below) $M_\pm$ is finite if $\lim_{\alpha \to \pm d_\pm} H(\pm \alpha) := H( d_\pm) < \infty$ and  $\lim_{\alpha \to \pm d_\pm} \pm H'(\pm \alpha) := \pm H'(d_\pm) < \infty$, and  for $M\ge M_{\pm}$ we have 
\begin{equation}
\label{eq:dplus}
\inf_{\alpha>0} \frac{H(\pm\alpha)+M}{\alpha} = \frac{H( d_{\pm})+M}{ \pm d_\pm}= \pm\mathbb{E}_{P_{ d_{\pm}}}[g] +  \frac{M - M_+}{\pm d_\pm} \,.
\end{equation}

\item If $\pm d_\pm = \infty$ and $M_\pm$ is finite then $g$ is $P$-a.s. bounded above/below and for $M \ge M_{\pm}$ we have 
\begin{equation}
\label{eq:gplus}
\inf_{\alpha>0} \frac{H(\pm\alpha)+M}{\alpha}
=\mathrm{ess\,sup}%
_{x\in\mathcal{X}}\{\pm g(x)\}\,.
\end{equation}
\end{enumerate}
\end{enumerate}
\end{proposition}

\begin{proof}
The proof of Proposition \ref{aux:bound:func:gen:prop} can be found in Appendix~\ref{app:a}.
A geometric depiction of the Proposition when $M<M_{\pm}$ is shown in Figure~\ref{prop:3:1:fig}(a).
\end{proof}

\begin{figure}[tbh]
\centering
\includegraphics[width=0.8\textwidth]{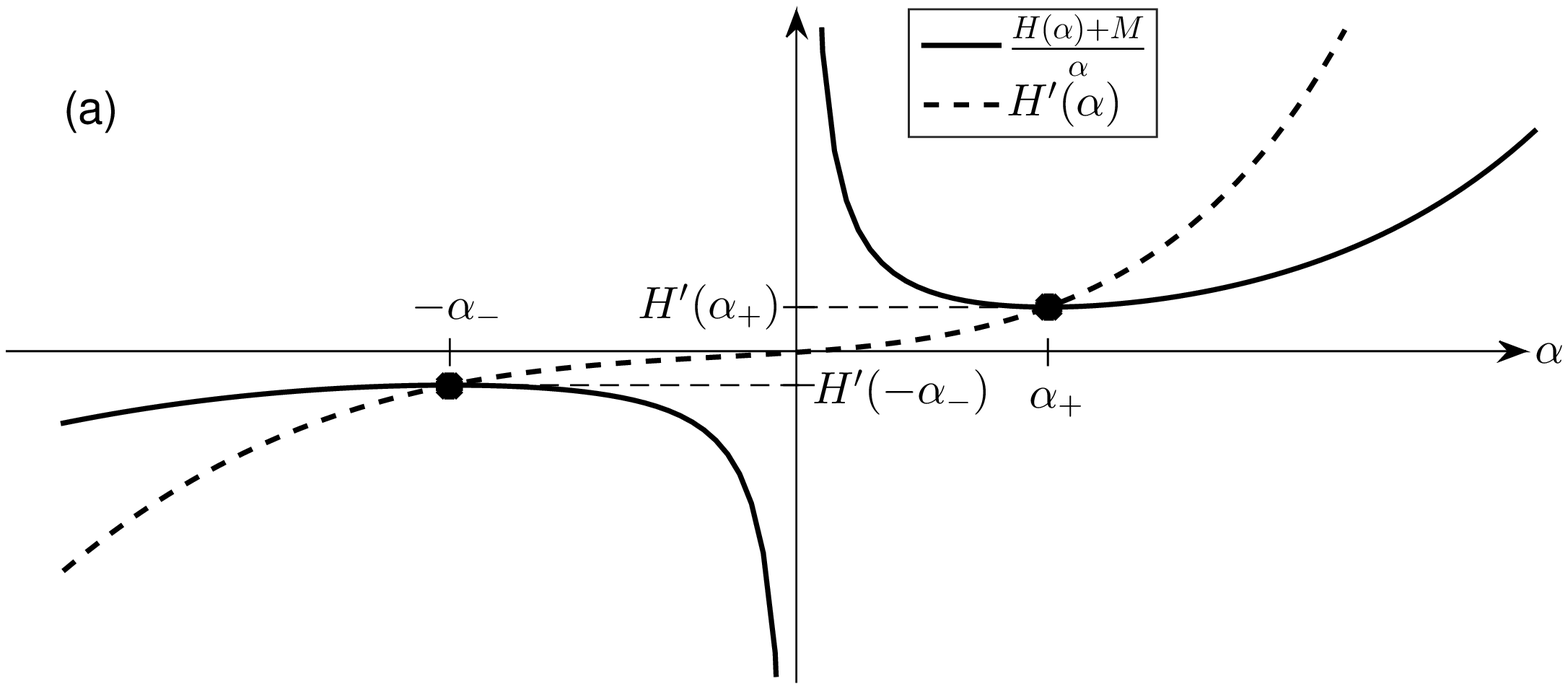}
\includegraphics[width=0.8\textwidth]{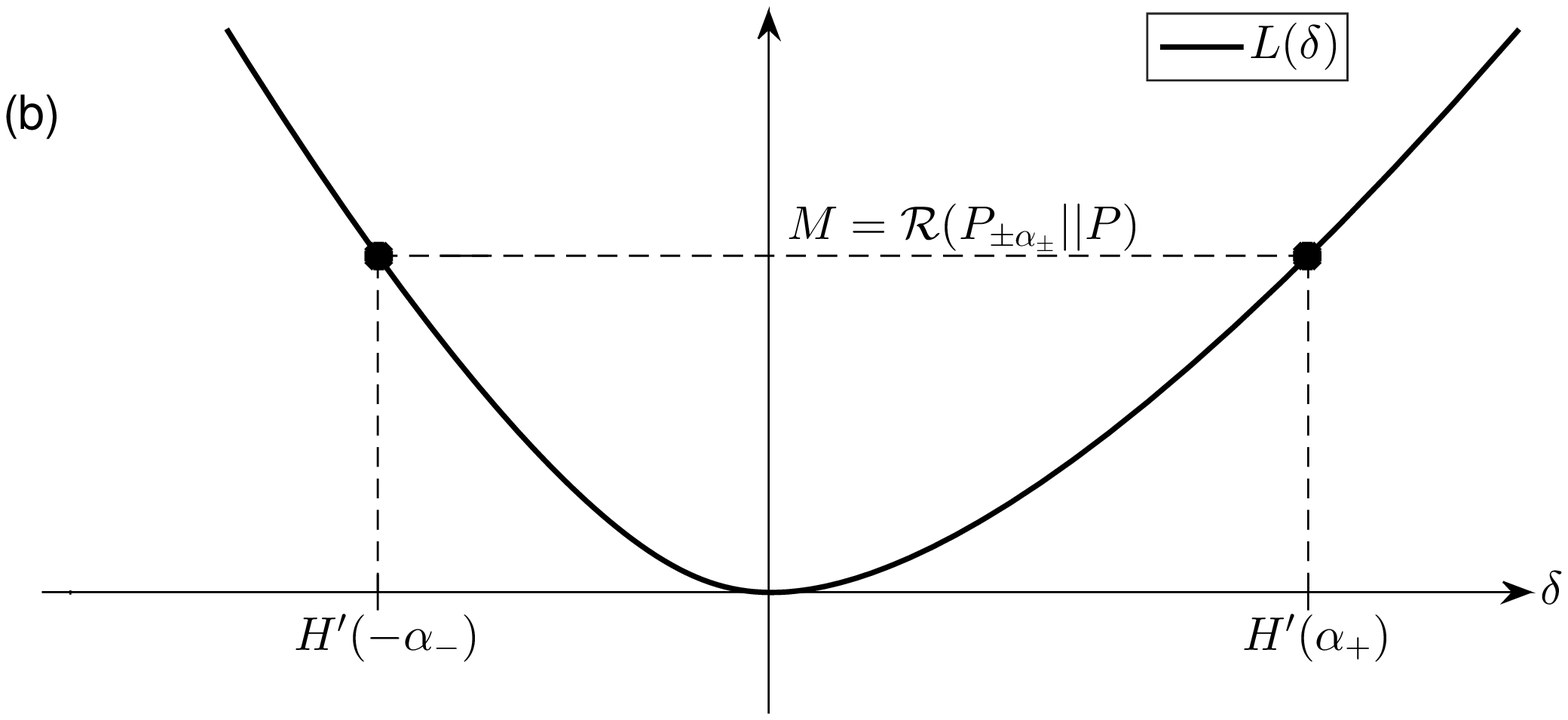}\caption{(a) Graphical representation of \VIZ{eq:Bpmeqn} which depicts the relation between the minimum of $(H(\a)+M)/\a$ (solid line) and the derivative of cumulant generating function, $H'(\a)$ (dashed line). Using the fact that $\inf_{\a>0}(H(-\a)+M)/\a=-\sup_{\a<0}(H(\a)+M)/\a$, we display the relation for $\a_-$ at the third quadrant. Note that both minimizers are attained at the intersections of the two curves, however, the two branches are generally not symmetric resulting in different values (i.e., $\a_+\neq \a_-$).
(b) Graphical representation of Remark \ref{Legendre:remark} which relates the optimal values of $(H(\a)+M)/\a$ with the Legendre transform of the cumulant generating function. For demonstration clarity, we assume that $\mathbb E[g]=0$
so that $L(\delta) \geq 0$ and $L(\delta)=0$ if and only if $\delta=0$.
}%
\label{prop:3:1:fig}%
\end{figure}

Unless the random variable $g$ is symmetric, in general the two 
optimization problems are not related to each other since they 
involve the  cumulant  generating function for $\alpha >0$ and 
$\alpha<0$ respectively.   

\begin{remark}
An alternative characterization of the minimizers uses the Legendre-Fenchel
transform of $H(\alpha)$,
\begin{equation*}
L(\delta)=\sup_{\alpha\in\mathbb{R}}\{\alpha\delta-H(\alpha)\}, \label{LT:gen}%
\end{equation*}
which is a convex function with a
mimimum equal to $0$ that is attained at $\mathbb{E}_{P}[g]$. The optimality condition
\eqref{eq:cond:re} for $\alpha_{\pm}$ can be written, equivalently, as
\begin{equation*}
M=L(\pm\delta_{\pm})
\mathrm{~~where~~}\delta_{\pm}= \pm H^{\prime}(\pm\alpha_{\pm}).
\end{equation*}
Thus the set of values of $M$ for which a finite minimizer exists 
corresponds to the possible level sets for  $L(\delta)$ 
(the  rate function in Cramer's theorem \cite{Dembo:98}), provided $L(\delta)$ is strictly convex. 
%
%
%
\label{Legendre:remark}
\end{remark}

If the function $g$ is centered, i.e., $\mathbb{E}_{P}[g]=0$, then when 
$M=0$ the minimizers are $\alpha_{\pm}=0$ and we can expand $\alpha_{\pm}$ as a 
Taylor series in the variable $\sqrt{M}$ as the following proposition shows,  
which was  proved in \cite[Lemma 2.10]{Dupuis:16} and will be useful for non-rare events. 

\begin{proposition}
[Linearization]\label{opt:bound:linear:gen:prop} If $g\in\mathcal{E}$ with
$\mathbb{E}_{P}[g]=0$ then we have
\begin{equation*}
\inf_{\alpha>0} \frac{H(\pm\alpha)+M}{\alpha}
= \sqrt{2\mathrm{Var}_{P}(g) M}%
+O(M)\ .\label{bound:linearization:eq}%
\end{equation*}

\end{proposition}

\subsection{R\'{e}nyi divergence, relative entropy, and a variational principle}
\label{sec:renyi}

In this section, we discuss the concepts of  R\'{e}nyi divergence
and associated variational representations and duality formulas. These tools provide the mathematical foundations for the uncertainty quantification and sensitivity analysis methods introduced in the subsequent sections.

Given $P,Q\in\mathcal{P}(\mathcal{X})$, we pick a reference measure
$R\in\mathcal{P}(\mathcal{X})$, such that $P\ll R$ and $Q\ll R$ (i.e., $P$ and
$Q$ are absolutely continuous with respect to $R$). Denote by $p=\frac{dP}%
{dR}$ (resp. $q=\frac{dQ}{dR}$) the Radon-Nikodym derivative
of $P$ (resp. $Q$) with respect to the reference measure $R$. Then, for
$\alpha>0,\alpha\neq1$, the R\'{e}nyi divergence of degree $\alpha$ of $Q$
with respect to $P$ is defined by (cf. \cite{Liese:87,Atar:15})
\begin{equation*}
\mathcal{R}_{\alpha}\left(  {Q}{\,||\,}{P}\right)  :=\left\{
\begin{array}
[c]{cl}%
\frac{1}{\alpha(\alpha-1)}\log\int_{pq>0}\left(  \frac{q}{p}\right)  ^{\alpha
}dP & \text{if\ }Q\ll P\text{\ or\ }\alpha<1\\
+\infty & \text{otherwise}%
\end{array}
\right. , \label{Renyi:def}%
\end{equation*}
and is independent of the choice of the reference measure $R$. Another common
definition of R\'{e}nyi divergence utilizes the factor $\frac{1}{\alpha-1}$
(cf. \cite{Renyi:61, Vajda:90, Erven:14}) instead of $\frac{1}%
{\alpha(\alpha-1)}$. When $P$ and $Q$ are mutually absolutely continuous,
R\'{e}nyi divergence can be written without reference to a reference measure
as
\begin{equation}
\mathcal{R}_{\alpha}\left(  {Q}{\,||\,}{P}\right)  =\frac{1}{\alpha(\alpha
-1)}\log\mathbb{E}_{P}\left[  \left(  \frac{dQ}{dP}\right)  ^{\alpha}\right]
=\frac{1}{\alpha(\alpha-1)}\log\mathbb{E}_{P}\left[  \exp\left\{  \alpha
\log\frac{dQ}{dP}\right\}  \right]  \ ,\label{Renyi:def:alt}%
\end{equation}
where the rightmost expression reveals that R\'{e}nyi divergence is
proportional to the cumulant (or log-moment) generating function for the
logarithm of the Radon-Nikodym derivative (i.e., $\log\frac{dQ}{dP}$). The
definition of R\'{e}nyi divergence is extended to $\alpha=1$ by letting
$\mathcal{R}_{\alpha}\left(  {Q}{\,||\,}{P}\right)  =\mathcal{R}\left(
{Q}{\,||\,}{P}\right)  $, where $\mathcal{R}\left(  {Q}{\,||\,}{P}\right)  $ is
the relative entropy (or Kullback-Leibler divergence) defined by
\begin{equation*}
\mathcal{R}\left( {Q}{\,||\,}{P}\right)  :=\left\{
\begin{array}
[c]{cl}%
\int_{pq>0}\frac{q}{p}\log\frac{q}{p}dP & \text{if\ }Q\ll P\\
+\infty & \text{otherwise}%
\end{array}
\right. , \label{rel:entr:def}%
\end{equation*}
and we have $\lim_{\alpha\rightarrow1}\mathcal{R}_{\alpha}\left(  {Q}%
{\,||\,}{P}\right)  =\mathcal{R}\left(  {Q}{\,||\,}{P}\right)$.
In view of \eqref{Renyi:def:alt} we can also extend  
$\mathcal{R}_{\alpha}\left( {Q}{\,||\,}{P}\right)$ to 
negative values of $\alpha$ if $Q$ and $P$  are mutually absolutely 
continuous.
Due to our convention for the definition of 
$\mathcal{R}_{\alpha}\left({Q}{\,||\,}{P}\right)$ we have the 
symmetry  $\mathcal{R}_{\alpha}\left( {Q}{\,||\,}{P}\right)= \mathcal{R}_{1- \alpha}\left( {P}{\,||\,}{Q}\right)$ and thus, in particular, $\mathcal{R}_{0}\left({Q}{\,||\,}{P}\right)=\mathcal{R}\left({P}{\,||\,}{Q}\right)$.
Further properties of the R\'{e}nyi 
divergence can be found for example in \cite{Golshani:09,Vajda:90, Erven:14}.

A variational formula in terms of R\'enyi divergence,  recently derived in \cite[Theorem 2.1]{Chowdhary:13}, will play a central role in this paper. 

\begin{theorem}[Variational representation involving R\'enyi divergence]
\label{thm:renyi:dupuis}
Let $\beta, \gamma \in \mathbb{R} \setminus\{0\}$ 
with  $\beta < \gamma$ and let $P \in \mathcal{P}(\mathcal{X})$.  For any bounded and measurable 
$f: {\mathcal X} \to \mathbb{R}$ we have 
\begin{eqnarray}
\frac{1}{\beta}\log\mathbb{E}_{Q}[e^{\beta f}]
&=&\inf_{P\in\mathcal{P}(\mathcal{X})}
\left\{   \frac{1}{\gamma}\log\mathbb{E}_{P}[e^{\gamma f}]+\frac{1}{\gamma - \beta}\mathcal{R}_{\frac{\gamma}{\gamma -\beta}}\left(  {Q}{\,||\,}{P}\right)
\right\}  \ , \label{risk:sens:inf}  \\
\frac{1}{\gamma}\log\mathbb{E}_{Q}[e^{\gamma f}]&=&\sup_{P\in\mathcal{P}%
(\mathcal{X})}\left\{  \frac{1}{\beta}\log\mathbb{E}_{P}[e^{\beta f}]-\frac{1}{\gamma - \beta}\mathcal{R}_{\frac{\gamma}{\gamma - \beta}}\left(  {P}{\,||\,}{Q}\right)
\right\}  \ . \label{risk:sens:sup}%
\end{eqnarray}
\end{theorem}

\begin{remark}The  variational representation formula
\eqref{risk:sens:sup}
is a generalization of the Donsker-Varadhan variational formula involving relative entropy (also know as the Gibbs variational principle) \cite{Dupuis:97,Simon:93}. Indeed taking  $\gamma=1$
and  $\beta \to 0$ in \eqref{risk:sens:sup} we 
obtain  the well-known formula
\begin{equation}
\log\mathbb{E}_{Q}[e^{f}]=\sup_{P\in\mathcal{P}%
(\mathcal{X})}\left\{ \mathbb{E}_{P}[f] - \mathcal{R}\left(  {P}{\,||\,}{Q}\right)
\right\}  \ . \label{donsker:varadhan}%
\end{equation}
Note that the variational 
formula \eqref{donsker:varadhan} 
serves as the basis of the UQ
theory for typical events developed in \cite{Chowdhary:13,Dupuis:16,Katsoulakis:17,Gourgoulias:17}. 

\end{remark}


\section{UQ Bounds for Rare Events}

\label{sec:UQ:bounds}

Let $P,Q\in\mathcal{P}(\mathcal{X})$ and $g:\mathcal{X}\to\mathbb{R}$ be a
measurable function. It is convenient to think of $Q$ as the ``true''
probabilistic model and of 
$P$ as a ``nominal'' or ``reference'' model. 
By Theorem \ref{thm:renyi:dupuis}, equation  \eqref{risk:sens:inf},
we have the upper bound 
\begin{equation}
\frac{1}{\beta} \log\mathbb{E}_{Q} \big[e^{\beta g}\big] \le\frac{1}{\gamma}
\log\mathbb{E}_{P} \big[e^{\gamma g}\big] + \frac{1}{\gamma-\beta}
\mathcal{R}_{\frac{\gamma}{\gamma-\beta}}\left( {Q}{\,||\,}{P}\right)\,,\label{risk:sens:bound}%
\end{equation}
which constitutes an upper bound for risk-sensitive observables (i.e., where tail
events matter). The upper bound consists of two terms; one being an estimate
of the risk-sensitive observable under the ``nominal'' model and the second
term being the cost to be paid for the substitution of the ``true'' model,
here, quantified by the R\'{e}nyi divergence.


By taking the limit $\beta\to0$, it formally holds that
\begin{equation}
\mathbb{E}_{Q} [ g] \le\frac{1}{\gamma} \log\mathbb{E}_{P} \big[e^{\gamma
g}\big] + \frac{1}{\gamma} \mathcal{R}\left( {Q}{\,||\,}{P}\right)\,,\label{typ:bound}%
\end{equation}
which is an upper bound for typical (i.e., not risk-sensitive) observables.
This upper bound \eqref{typ:bound} was
the starting point for deriving UQ and sensitivity
bounds for typical observables in \cite{Chowdhary:13, Li:12,Dupuis:16,Katsoulakis:17,Gourgoulias:17}. 
Here we derive respective bounds for log-probabilities of rare
events, which form a particular class of risk-sensitive observables.
In other words, we are interested in bounding quantities of the form $\log Q(A) - \log P(A)$.
The following theorem summarizes the UQ bounds for rare events.

\begin{theorem}
\label{UQ:bounds:lemma} (a) Fix $P,Q\in\mathcal{P}(\mathcal{X})$ and let
$A\in\mathcal{B}(\mathcal{X})$ be such that $P(A)>0$ and $Q(A)>0$. Then
\begin{equation}
\sup_{\alpha>0} \left\{  -(\alpha+1)\mathcal{R}_{\alpha+1}\left( {P}%
{\,||\,}{Q}\right)  + \frac{1}{\alpha} \log P(A) \right\}  \le\log Q(A) - \log
P(A) \le\inf_{\alpha>0} \left\{  \alpha\mathcal{R}_{\alpha+1}\left(
{Q}{\,||\,}{P}\right)  - \frac{1}{\alpha+1} \log P(A) \right\}  \,.
\label{opt:rare:event:bound:a}%
\end{equation}

(b) If $P$ and $Q$ are mutually absolutely continuous, then
(\ref{opt:rare:event:bound:a}) can be rewritten as {\small
\begin{equation}
- \inf_{\alpha>0} \left\{  \frac{\log\mathbb{E}_{P}\left[ \left(  \frac{d Q}{d
P}\right) ^{-\alpha}\right]  -\log P(A)}{\alpha} \right\}  \le\log Q(A) - \log
P(A) \le\inf_{\alpha>1} \left\{  \frac{\log\mathbb{E}_{P}\left[  \left(
\frac{d Q}{d P}\right) ^{\alpha}\right] -\log P(A)}{\alpha} \right\}
\ .\label{opt:rare:event:bound2:a}%
\end{equation}
}
\end{theorem}

\begin{proof}
(a) We first prove the upper bound. By setting $\b=\a$ and $\g=\a+1$ in \VIZ{risk:sens:bound}, we get
\begin{equation*}
\frac{1}{\alpha} \log\mathbb{E}_{Q}[e^{\alpha g}] \leq
\frac{1}{\alpha+1}\log\mathbb{E}_{P}[e^{(\alpha+1)g}] + \RENYI{\alpha+1}{Q}{P} \ .
\end{equation*}
By taking $\alpha >0$ and  considering $g=0$ on $A$, $g=-M$ on $A^c$ and then sending $M\to\infty$, the above inequality
becomes
\begin{equation*}
\frac{1}{\alpha} \log Q(A) \leq
\frac{1}{\alpha+1}\log P(A) + \RENYI{\alpha+1}{Q}{P} \ .
\end{equation*}
Multiplying with $\a$ and subtracting $\log P(A)$, we have
\begin{equation*}
\log Q(A) - \log P(A) \leq
\a\RENYI{\alpha+1}{Q}{P} - \frac{1}{\alpha+1}\log P(A) \ .
\end{equation*}
Since this holds for all $\a>0$, the upper bound in \VIZ{opt:rare:event:bound:a} is proved.
For the lower bound, we reverse $Q$ and $P$ in \VIZ{risk:sens:bound} and proceed
as in the upper bound.

\noindent
(b) Substituting the R\'{e}nyi formula \VIZ{Renyi:def:alt} in \VIZ{opt:rare:event:bound:a}, we get{\small
\begin{equation*}
\sup_{\alpha>0} \left\{  -\frac{1}{\alpha} \log\mathbb{E}_{Q}\left[\left(  \frac{d P}{d Q}\right)^{\a+1}\right]  + \frac{1}{\a}\log P(A) \right\}
\le \log Q(A) - \log P(A) \le
\inf_{\alpha>0} \left\{  \frac{1}{\a+1}\log\mathbb{E}_{P}\left[  \left(  \frac{d Q}{d P}\right)^{\a+1}\right] - \frac{1}{\a+1} \log P(A) \right\}  \ .
\end{equation*}}
Equivalently, {\small
\begin{equation*}
\sup_{\alpha>0} \left\{  -\frac{1}{\alpha} \log\mathbb{E}_{P}\left[\left(  \frac{d P}{d Q}\right)^{\a}\right]  + \frac{1}{\a}\log P(A) \right\}
\le \log Q(A) - \log P(A) \le
\inf_{\alpha>1} \left\{  \frac{1}{\a}\log\mathbb{E}_{P}\left[  \left(  \frac{d Q}{d P}\right)^{\a}\right] - \frac{1}{\a} \log P(A) \right\}  \ .
\end{equation*}}
which is exactly \VIZ{opt:rare:event:bound2:a}.
\end{proof}

We turn next to determining the optimal $\alpha$ in
\eqref{opt:rare:event:bound2:a} using the results from Section
\ref{sec:opt:problem}. We select the function $g=\log\frac{dQ}{dP}%
\in\mathcal{E}$, whose cumulant generating function is
\begin{equation}
H(\alpha)=\log\mathbb{E}_{P}[e^{\alpha\log\frac{dQ}{dP}}]\ .\label{uq:cgf:def}%
\end{equation}
As is readily apparent from the bound
\eqref{opt:rare:event:bound2:a}, in order to obtain non-trivial upper and lower bounds we should assume 
$H(\alpha)$ is finite in an open neighborhood of the interval $[0,1]$. If we assume this  
then the function $H(\alpha)$ has the following elementary properties:

\begin{enumerate}
\item $H(0)=H(1)=0$.

\item $H^{\prime}(0)=-\mathcal{R}\left(  {P}{\,||\,}{Q}\right)  $, $H^{\prime
}(1)=\mathcal{R}\left(  {Q}{\,||\,}{P}\right)  $.

\item More generally for any $\alpha$ we have
\begin{equation}
\label{uq:cgf:deriv}H^{\prime}(\alpha) \,=\, \mathbb{E}_{P_{\alpha}} \left[
\log\frac{d Q}{d P}\right]  = \mathcal{R}\left( {P_{\alpha}}{\,||\,}%
{P}\right)  - \mathcal{R}\left( {P_{\alpha}}{\,||\,}{Q}\right) \ ,
\end{equation}
where $P_{\alpha}$ is the exponential family given by
\begin{equation}
\label{uq:cgf:tilted}\frac{dP_{\alpha}}{dR} \,=\, \frac{ q^{\alpha}%
p^{1-\alpha}}{ \int q^{\alpha}p^{1-\alpha} dR} \,.
\end{equation}
The family $P_{\alpha}$ interpolates between $P$ and $Q$ since $P_{0}=P$ and
$P_{1}=Q$.

%
%
%

\end{enumerate}

\noindent
Using these properties, as well as  Proposition \ref{aux:bound:func:gen:prop} we obtain the
following, explicit UQ bounds for all rare events $A$ such that $-\log P(A)=M$.

{

\begin{theorem}[{UQ bounds for rare events}]\label{thm:uq:rare}
Let $P,Q\in\mathcal{P}(\mathcal{X})$ be mutually absolutely continuous and assume that $H(\a)$ given in \eqref{uq:cgf:def}
is finite for $\alpha$ in a neighborhood of $[0,1]$.   Let $M_{\pm}$,  $0<M_{\pm}\leq\infty$,  be the constants 
given in Proposition \ref{aux:bound:func:gen:prop} (with $g=\log\frac{dQ}{dP}$).   For any  $A \in\mathcal{B}(\mathcal{X})$ with  
$P(A) = e^{-M}>0$ with $M \le M_\pm$ we have 
\begin{equation}
-\mathcal{R}\left(  {P_{-\alpha_{-}}}{\,||\,}{Q}\right)  \,\leq\,\log
Q(A)\,\leq\,  \left\{ 
		            \begin{array}{cl}  
                                           0 & \textrm{ if } M <  \mathcal{R}(Q\,||\,P) \\          
                                         -  \mathcal{R}\left(  {P_{\alpha_{+}}}{\,||\,}{Q}\right) &  \textrm{ if } M \ge   \mathcal{R}(Q\,||\,P)
                                \end{array}
                    \, \right.
              \label{UQ:bound:rep2}.      
\end{equation}
where $\alpha_{\pm}=\alpha_\pm(M)$ are the (unique) solutions of
\begin{equation}
\mathcal{R}\left(  {P_{\pm}\alpha_{\pm}}{\,||\,}{P}\right)  =M=-\log
P(A)\, .\label{UQ;bound:rep2}%
\end{equation}
\end{theorem}

\begin{proof}  With $H(\alpha)$ given in \VIZ{uq:cgf:def} and setting $M=-\log P(A)$, part (b) of Theorem~\ref{UQ:bounds:lemma} is
rewritten as
\begin{equation}\label{eq:UQbound:final}
- \inf_{\a>0} \left\{ \frac{H(-\a)+M}{\a} \right\} \le \log Q(A) + M \le \inf_{\a>1}\left\{\frac{H(\a)+M}{\a}\right\}\,.
\end{equation}
The lower bound is then an immediate consequence of Proposition \ref{aux:bound:func:gen:prop}  together with \eqref{uq:cgf:deriv}
and \eqref{UQ;bound:rep2}. 

However, the upper bound involves a modified  calculation since the infimum is taken only over $\alpha>1$.  We first note that the corresponding minimization of the quantity 
\begin{equation}
B(\alpha;M)=\frac{H(\alpha)+M}{\alpha}\, ,
\end{equation}
arising also in  the proof of Proposition \ref{aux:bound:func:gen:prop}, e.g. \eqref{eq:B:Appendix}, separates into two distinct cases. Indeed,   for any $\alpha >1$ we have
\begin{equation}\label{eq:UQbound:final:0}
B^{\prime}(\alpha;M)=\frac{\alpha H^{\prime}(\alpha)-H(\alpha)-M}{\alpha^{2}}
> \frac{1 \cdot  H^{\prime}(1)-H(1)-M}{\alpha^{2}}=\frac{\mathcal{R}\left(  {Q}{\,||\,}{P}\right)-M}{\alpha^2}\, , 
\end{equation}
using the Properties 1-3 above for $H(\alpha)$ and the fact that $\alpha H^{\prime}(\alpha)-H(\alpha)$ is a strictly increasing function of $\a$ in $[0,\infty)$.

First, \eqref{eq:UQbound:final:0} implies that  if 
\begin{equation}
\label{eq:UQbound:threshold}
\mathcal{R}\left(  {Q}{\,||\,}{P}\right) > M=-\log P(A)\, ,
\end{equation}
 then $B(\alpha;M)$ is strictly increasing in $\alpha$, hence the infimum on the upper bound of \eqref{eq:UQbound:final} occurs at $\alpha=1$. In this case,  using that $H(1)=0$, we obtain that 
\begin{equation}\label{eq:UQbound:final:1}
\inf_{\a>1}\left\{\frac{H(\a)+M}{\a}\right\}=M\, .
\end{equation}

On the other hand,  if we have
\begin{equation}
\label{eq:UQbound:threshold2}
\mathcal{R}\left(  {Q}{\,||\,}{P}\right) \le  M=-\log P(A)\, ,
\end{equation}
%
then $B(\alpha; M)$ has a unique minimum,  and
the minimizer (equivalently, the minimizer of the upper bound of \eqref{eq:UQbound:final}) occurs  at the unique finite root $\alpha_+=\alpha_+(M)>1$, namely 
\begin{equation}\label{eq:UQbound:final:2}
\inf_{\a>1}\left\{\frac{H(\a)+M}{\a}\right\}=H'(\alpha_+)=M-\mathcal{R}(P_{\alpha_+}\,||\,Q)\, .
\end{equation}
We now combine \eqref{eq:UQbound:final:1} and \eqref{eq:UQbound:final:2} with the upper bound of 
\eqref{eq:UQbound:final}
to obtain \eqref{UQ:bound:rep2}.

Note that the upper bound in  \eqref{UQ:bound:rep2} is not discontinuous in $M$ since for $M=\mathcal{R}(P\,||\,Q)$,  $\alpha_+=1$
and by Property 3 we have that  $P_1=Q$;  thus $\mathcal{R}(P_{\alpha_+}\,||\,Q)=\mathcal{R}(Q\,||\,Q)=0$.  
%

\end{proof}

\begin{remark} We obtain the trivial bound $\log Q(A)\le 0$ in Theorem \ref{thm:uq:rare}  when the true measure $Q$ 
is (relatively) too far  from the reference model $P$. 
%
This relative ``distance''  (see  \eqref{eq:UQbound:threshold} and \eqref{eq:UQbound:threshold2})  
is quantified by the ratio $\frac{\mathcal{R}(Q\,||\,P)}{-\log P(A)}$ which  is required to be less than  $1$ to obtain an 
informative upper bound in \eqref{UQ:bound:rep2}.
\end{remark}
}



\begin{remark}
We can interpret the condition for $\alpha_{\pm}$ in (\ref{UQ;bound:rep2}) by
noting that the measure $P$ conditioned on the rare event $A$, $P_{|A}$
satisfies  $\mathcal{R}\left(  {P_{|A}}{\,||\,}{P}\right)  =-\log P(A) \ .$ 
Theorem \ref{thm:uq:rare}  states that one should find the proper mixtures of $P$ and $Q$ (as
described by $P_{\pm\alpha_{\pm}}$) so that   $ \mathcal{R}\left(  {P_{\pm \alpha_\pm}} {\,||\,} {P}\right) = \mathcal{R}\left(  {P_{|A}}{\,||\,}{P}\right)$.   The bounds for $-\log Q(A)$ are then simply $\mathcal{R}\left(  {P_{\pm\alpha_\pm }}{\,||\,}{Q}\right)  $.
\end{remark}

Finally, we note that a much cruder UQ bound than
\eqref{opt:rare:event:bound:a} can be obtained by considering the upper lower bounds obtained by taking $\alpha=\infty$
in \eqref{opt:rare:event:bound:a}. Indeed, we can consider the alternative definition of R\'{e}nyi divergence \cite{Erven:14}
$$\mathcal{D}_\alpha(Q||P)=\alpha \mathcal{R}_{\alpha}(Q||P)\, ,$$
and accordingly rewrite \eqref{opt:rare:event:bound:a} for $\mathcal{D}_\alpha(Q||P)$.
Noting that
$$
\mathcal{D}_\infty(Q||P)=\sup_{x\in\mathcal{X}}{\log\frac{dQ}{dP}}\, ,
$$
also referred as worst-case regret, \cite{Erven:14},
we can bound \eqref{opt:rare:event:bound:a} from above and below by selecting $\alpha=\infty$.
This substitution obviously yields  a less sharp version of \eqref{opt:rare:event:bound:a}, namely the (trivial) bound
\begin{equation}
 \inf_{x\in\mathcal{X}}{\log\frac{dQ}{dP}}\,= -\mathcal{D}_{\infty}\left( {P}%
{\,||\,}{Q}\right)    \le\log Q(A) - \log
P(A) \le  \mathcal{D}_{\infty}\left(
{Q}{\,||\,}{P}\right)  =
\,\sup_{x\in\mathcal{X}}{\log\frac{dQ}{dP}}  \,,
\label{opt:rare:event:bound:infty}%
\end{equation}
Note that this bound is valid if $M_{\pm}<\infty$ and $M>M_{\pm}$, as long as $\log\frac{dQ}{dP}$ is bounded from above/below. 

\section{Sensitivity Indices for Rare Events}

\label{sec:Sens:bounds}

In this section we consider a parametric family of probability measures
$P^{\theta}\in\mathcal{P} (\mathcal{X})$ with $\theta\in\mathbb{R}^{K}$ and we
assume $P^{\theta}\ll R$, where $R$ is a reference measure in $P(\mathcal{X})$
and with density $p^{\theta} =\frac{dP^{\theta}}{dR}$. We further assume that
the mapping $\theta\mapsto p^{\theta}(x)$ is twice differentiable with respect
to $\theta$ for all $x\in\mathcal{X}$ together with a suitable integrability
condition on $\log p^{\theta}(x)$ to allow the interchange of integral and
derivatives. The sensitivity index for a rare event $A\in\mathcal{B}%
(\mathcal{X})$ (with $P(A)>0$) in the direction $v\in\mathbb{R}^{K}$ is then
given by
\begin{equation}
S_{v}^{\theta}(A):=\lim_{\epsilon\rightarrow0}\frac{\log P^{\theta+\epsilon
v}(A)-\log P^{\theta}(A)}{\epsilon} \,=\, \mathbb{E}_{P^{\theta}_{\vert A}} [
v^{T} W^{\theta}] \,=\, \frac{\mathbb{E}_{P^{\theta}}[\chi_{A}v^{T} W^{\theta
}]} {P(A)} \ ,\label{SI2}%
\end{equation}
where $W^{\theta}= \nabla_{\theta}\log p^{\theta}$ is the score of the
probability measure $P^{\theta}$, see also \eqref{LRR}.

Here, we derive computationally tractable bounds on the sensitivity indices $S_{v}^{\theta}(A)$ and corresponding new  rare event sensitivity indices $\mathcal{I}_{v,\pm}^{\theta}$, 
starting from the UQ bounds presented in Section \ref{sec:UQ:bounds}. 

By considering
the measures $Q=P^{\theta+\epsilon v}$ and $P=P^{\theta}$ we have that
$\log\frac{dP^{\theta+\epsilon v}}{dP^{\theta}} = O(\epsilon)$ and thus it is
natural to rescale the parameter $\alpha$ according to 
\begin{equation}
\alpha=\frac{\alpha_{0}}{\epsilon}\ .\label{alpha:expansion}%
\end{equation}
Taking $\epsilon\to0$ we obtain a non-trivial bound for the sensitivity
indices as the following Theorem  shows.
Next, in order to state our results we require the cumulant generating function
for the score function $W^{\theta}$ defined in  \eqref{score}:
\begin{equation*}\label{sec5:H:definition}
H_{v}^{\theta}(\alpha)=\log\mathbb{E}_{P^{\theta}}\left[  e^{\alpha
v^{T}W^{\theta}}\right]  \,.
\end{equation*}
This cumulant generating function has the following elementary properties:
\begin{enumerate}
\item $(H_{v}^{\theta})^{\prime}(0) = \mathbb{E}_{P^{\theta}}[ v^{T}
W^{\theta}] =0$.

\item $(H_{v}^{\theta})^{\prime\prime  }(0)= v^T\mathbb{E}_{P^{\theta}}[ W^{\theta
}(W^{\theta})^{T}] v = v^{T} \mathcal{F}(P^{\theta}) v$ where $\mathcal{F}%
(P^{\theta})$ denotes the Fisher information matrix for the parametric family
$P^{\theta}$.

\item More generally we have
\begin{equation*}
(H_{v}^{\theta})^{\prime}(\alpha)=\mathbb{E}_{P_{\alpha}^{\theta}}%
[v^{T}W^{\theta}]\label{sens:cgf:deriv} \,, %
\end{equation*}
where $P_{\alpha}^{\theta}$ is the exponential family with
\begin{equation*}
\frac{dP_{\alpha}^{\theta}}{dP^{\theta}}\,=\,e^{\alpha v^{T}W^{\theta}%
-H_{v}^{\theta}(\alpha)}\,.\label{sens::cgf:tilted}%
\end{equation*}

\end{enumerate}


\begin{theorem}\label{sec5:thm1}
Assume that the mapping $\theta\mapsto p^{\theta}(x)$ is $\mathcal{C}^{2}$ for
all $x\in\mathcal{X}$ and that for each $\alpha_{0}>0$ there exists $\delta>0$
such that
\[
\mathbb{E}_{P^{\theta}}\Big[e^{\alpha_{0}v^{T}\nabla_{\theta}\log p^{\theta
}+\alpha_{0}\frac{\delta}{2}\sup_{|\theta-\theta^{\prime}|<\delta}|v^{T}%
\nabla_{\theta}^{2}\log p^{\theta^{\prime}}v|}\Big]<\infty\ .
\]
Then
\begin{equation}
-\inf_{\alpha>0}\left\{  \frac{H_{v}^{\theta}(-\alpha)-\log P^{\theta}%
(A)}{\alpha}\right\}  \leq S_{v}^{\theta}(A)\leq\inf_{\alpha>0}\left\{
\frac{H_{v}^{\theta}(\alpha)-\log P^{\theta}(A)}{\alpha}\right\} \,.
\label{sens:bound:gen}%
\end{equation}

\end{theorem}

\begin{proof}
Rewriting \VIZ{opt:rare:event:bound2:a} for $P^{\8+\e v}$ and $P^{\8}$ and substituting $\a$ from \VIZ{alpha:expansion},
we get the upper bound
\begin{equation*}
\log P^{\8+\e v}(A) - \log P^\8(A) \le
\left\{  \frac{\log\mathbb{E}_{P^\8}\left[  e^{\frac{\a_0}{\e} \left(\log p^{\8+\e v} - \log p^{\8} \right)}\right]-\log P^\8(A)}{\a_0/\e} \right\}  \ ,
\end{equation*}
valid for all $\a_0>\e$.
Dividing by $\e$, sending $\e\to0$, and then infimizing on $\alpha_0$, we get
\begin{equation*}
\lim_{\e\to0} \frac{1}{\e} \big(\log P^{\8+\e v}(A) - \log P^\8(A) \big) \le
\inf_{\a_0>0} \left\{  \frac{\lim_{\e\to0}\log\mathbb{E}_{P^\8}\left[  e^{\frac{\a_0}{\e}\left(\log p^{\8+\e v} - \log p^{\8} \right)}\right]-\log P^\8(A)}{\a_0} \right\}  \ .
\end{equation*}
In order to complete the proof of the upper bound we have to interchange between the limit and the integral.
This is justified by the dominated convergence theorem since we have  from Taylor's theorem that
\begin{equation*}
e^{\a_0 \frac{1}{\e} \left(\log p^{\8+\e v} - \log p^{\8} \right)}
= e^{\a_0 \left( v^T \nabla_\8 \log p^{\8} + \frac{\e}{2} v^T \nabla_\8^2 \log p^{\8'} v \right)}
\leq e^{\a_0 \left( v^T \nabla_\8 \log p^{\8} + \frac{\d}{2}  \sup_{|\8-\8'|<\d} |v^T\nabla_\8^2 \log p^{\8'} v | \right)}
\end{equation*}
for some $\8'$ in the interval defined by the points $\8$ and $\8+\e v$ and for $\e<\d$. Therefore,
\begin{equation*}
S_v^\8(A) \leq \inf_{\alpha_0>0} \frac{\log\mathbb{E}_{P^\8}\left[  e^{\a_0v^T \nabla_\8 \log p^{\8}}\right] -\log P^\8(A)}{\a_0}
\end{equation*}
which establishes the upper bound in \eqref{sens:bound:gen}.
The lower bound in \eqref{sens:bound:gen} is proved similarly.
\end{proof}

Using (\ref{eq:Bpmeqn}) and (\ref{eq:cond:re}) from Proposition
\ref{aux:bound:func:gen:prop} to evaluate the infimum in (\ref{sens:bound:gen}%
), we obtain a representation of the bounds.
%

\begin{theorem}[Sensitivity indices for rare events]
\label{thm:index}
\label{SB:bound:expl:thm}
Under the same assumptions as in Theorem \ref{sec5:thm1}
consider the family
\begin{equation*}
\mathcal{A}_{M}=\{ A\,:\, P^{\theta}(A) \ge e^{-M} \}
\end{equation*}
of all events $A$ which are less rare than a specified threshold $e^{-M}$. Then there exists $M_\pm$ such that for $M < M_\pm$ and any 
$A\in \mathcal{A}_{M}$ we have 
\begin{equation}\label{sens:bound:gen:0}
\mathcal{I}_{v,-}^{\theta}(M) \le
S_{v}^{\theta}(A)
\le
\mathcal{I}_{v,+}^{\theta}(M)\,,
\end{equation}
where 
\begin{equation}\label{SB:bound:expl}
\mathcal{I}_{v,\pm}^{\theta}(M):= \pm  \inf_{\alpha>0}\left\{
\frac{H_{v}^{\theta}(\pm \alpha)+M}{\alpha}\right\} 
= \mathbb{E}_{P^{\theta}_{\pm\alpha_\pm}}  \left[ v^{T}W^{\theta} \right] 
\end{equation}
and  $\alpha_\pm$ are determined by 
\begin{equation*}
\mathcal{R} ( { P^{\theta}_{\pm \alpha_{\pm}}}{\,||\,}{P^{\theta}}) =  M \,.
\end{equation*}
\end{theorem}

Similarly to Theorem~\ref{thm:index},  the  rare event sensitivity indices $\mathcal{I}_{v,\pm}^{\theta}(M)$ characterize the sensitivity of the model $P^\theta$ for each $M$-level set 
$\bar{\mathcal{A}}_{M}:=\{ A: \log P^{\theta}(A)=-M\}$, i.e. corresponding to all events which are equally rare and characterized by $M$.

The new sensitivity indices defined in \eqref{SB:bound:expl} are in general less sharp than the
gradient-based indices  $S_{v}^{\theta}(A)$ in \eqref{SI2},
due to the inequalities \eqref{sens:bound:gen:0}.
However, they do not require a rare event sampler  for {\em each} rare  event $A$, as one readily sees 
by comparing \eqref{LRR} to \eqref{SB:bound:expl}. In fact the indices $\mathcal{I}_{v,\pm}^{\theta}(M)$
are identical for the entire classes of rare events  in $\mathcal{A}_M$ or in $\bar{\mathcal{A}}_{M}$.
In this sense, they present similar computational advantages and  trade-offs as other
sensitivity indices for typical observables (not rare event-dependent), such as Fisher
information  bounds, \cite{Dupuis:16}; in particular, 
they are less sharp  but  can be  used to efficiently screen out  insensitive parameters in  directions $v$ in parameter space, i.e directions $v$ where 
$
\mathcal{I}_{v,\pm}^{\theta}(M)= \mathbb{E}_{P^{\theta}_{\pm\alpha_\pm}}  \left[ v^{T}W^{\theta} \right] \approx 0 
$;  we refer for such sensitivity screening results to  \cite{arampatzis2015}, at least  for typical events and observables.
%

\section{Bounds and Approximations for the Rare Event Sensitivity Indices}
\label{sec:estimation}

{ The upper and lower bounds  in Theorem~\ref{sec5:thm1}
and the representation of the sensitivity indices in Theorem~\ref{thm:index} suggest at least two approaches to practically implement  the indices $\mathcal{I}_{v,\pm}^{\theta}(M)$. The first one is based  on concentration inequalities, while the second one  relies on the direct statistical and numerical estimation of  the indices 
$\mathcal{I}_{v,\pm}^{\theta}(M)$. We next discuss the first approach.}


%
{

\medskip
\noindent
\textbf{Concentration Inequalities.} 
Concentration inequalities, i.e., explicit bounds of the probability of tail events, are often obtained via a Chernoff bound by 
using computable upper bounds on the cumulant generating function of the random variable.  Such upper bounds 
typically involve only a few  features of the underlying random variable such as mean, variance, bounds, higher moments, and so on, see for instance \cite{Boucheron:2016}. 

%
%
%

Here we can naturally use such inequalities to  provide simplified and computable bounds for  the variational formula for the sensitivity index, namely
$$
\mathcal{I}_{v,\pm}^{\theta}(M):= \pm  \inf_{\alpha>0}\left\{
\frac{H_{v}^{\theta}(\pm \alpha)+M}{\alpha}\right\} 
 \, ,
$$
by bounding  the cumulant  generating function of the score function 
$
H_{v}^{\theta}(\alpha)=\log\mathbb{E}_{P^{\theta}}\left[  e^{\alpha
v^{T}W^{\theta}}\right].
$
We provide  two such examples, by making the assumption  that the score function is bounded.  One can prove similar results in 
the same spirit  by using different assumptions on the tail behavior  of the score function, e.g. if we assume that the score $v^{T}W^{\theta}$ is a sub-Gaussian or a sub-Poissonian random variable \cite{Boucheron:2016}.

A similar use of various concentration inequalities in order to obtain  computable uncertainty quantification bounds 
for ordinary  observables was proposed recently in \cite{Gourgoulias:17}.

%
%

\begin{theorem}[Bernstein sensitivity bounds for rare events]
\label{thm:Bernstein}
We consider the  same assumptions  as in Theorem \ref{sec5:thm1}; we further assume 
$$
\sup_{x\in\mathcal{X}} v^{T}W^{\theta}(x) \le b_v
$$
for some $b_v >0$.  Furthermore, let 
$$
\mathcal{F}(P^{\theta}) = \mathbb{E}_{P^{\theta}}[ W^\theta (W^\theta)^T] 
$$ 
be  the Fisher information matrix for the parametric family 
$P^{\theta}$. Then, we have the following two bounds on the cumulant  generating function of the score function \eqref{score}, $H_{v}^{\theta}(\alpha)$, and the sensitivity index $\mathcal{I}_{v,+}^{\theta}(M)$:

\medskip
\noindent
(a) For all $\alpha>0$ we have the concentration inequality:
\begin{equation}
\label{eq:conc:ineq:1}
H_{v}^{\theta}(\alpha)=\log\mathbb{E}_{P^{\theta}}\left[  e^{\alpha
v^{T}W^{\theta}}\right] \le \frac{v^T\mathcal{F}(P^{\theta})v}{b_v^2}\phi(b_v\alpha)
\, , 
\end{equation}
where $\phi(x)=e^x-x-1$.

\medskip
\noindent
(b) Using the notation of Theorem~\ref{SB:bound:expl:thm}, we have for all $M < M_\pm$ :
\begin{equation}\label{eq:conc:bound:Bernstein}
\mathcal{I}_{v,+}^{\theta}(M)
\le 
b_vM+\sqrt{2v^T\mathcal{F}(P^{\theta})vM}
\,.
\end{equation}


\end{theorem}

\begin{proof} (a) Follows immediately from Theorem 2.9 in \cite{Boucheron:2016} by noting that 
$\mathbb{E}_{P^{\theta}}W^\theta=0$, and therefore we have that 
$\mbox{Var}_{P^{\theta}}(v^TW^\theta)=v^T\mathbb{E}_{P^{\theta}}[ W^\theta (W^\theta)^T]v=v^T\mathcal{F}(P^{\theta})v$.

\medskip
\noindent
(b) First we note that $\phi(x)\le\frac{x^2}{2(1-x)}$, for $0\le x<1$. Therefore, part (a) implies that
\begin{equation}
\label{eq:concentration:practical}
\mathcal{I}_{v,+}^{\theta}(M):=   \inf_{\alpha>0}\left\{
\frac{H_{v}^{\theta}(\alpha)+M}{\alpha}\right\} \le \inf_{0<b_v\alpha<1}\left\{
v^T\mathcal{F}(P^{\theta})v
\frac{\alpha^2}{2(1-b_v\alpha)}
+\frac{M}{\alpha}\right\}
 \, .
\end{equation}
We next change variables to $c=(b_v\alpha)^{-1}>1$, and then it is easy to show that the minimum occurs at $c^*=1+b^{-1}\sqrt{\frac{v^T\mathcal{F}(P^{\theta})v}{2M}}$. We conclude by substituting into the right hand side of
\eqref{eq:concentration:practical}.
\end{proof}



\begin{remark}
[A Bennett sensitivity bound for rare events]
A much tighter concentration inequality than \eqref{eq:conc:ineq:1} is given by the Bennett inequality in \cite{Dembo:98}, Lemma 2.4.1; see also Figure~\ref{ci:bounds:fig}.
Here, using again that $\mathbb{E}_{P^{\theta}}W^\theta=0$  and 
$\mbox{Var}_{P^{\theta}}(v^TW^\theta)=v^T\mathcal{F}(P^{\theta})v$, we have that
\begin{equation}
\label{eq:conc:ineq:2}
H_{v}^{\theta}(\alpha)=\log\mathbb{E}_{P^{\theta}}\left[  e^{\alpha
v^{T}W^{\theta}}\right] 
\le 
\log\left(\frac{{b_v}^2}{{b_v}^2+\sigma_v^2}\exp(-\alpha\sigma_v^2/{b_v})+\frac{\sigma_v^2}{{b_v}^2+\sigma_v^2}\exp(\alpha{b_v})\right)\, ,
\end{equation}
for all $\alpha\geq 0$ and where $\sigma_v^2$ is any upper bound of $\mbox{Var}_{P^{\theta}}(v^TW^\theta)=v^T\mathcal{F}(P^{\theta})v$.
Therefore, we can pick 
$$\sigma_v^2:=v^T\mathcal{F}(P^{\theta})v\quad \mbox{and}\quad 
b_v:=\sup_{x\in\mathcal{X}} v^{T}W^{\theta}(x)\, .$$
Then, the corresponding bound on $\mathcal{I}_{v,+}^{\theta}(M)$ is not analytically tractable, however the resulting variational representation is one dimensional and is trivial to find the optimal solution numerically:
\begin{equation}
\label{eq:conc:bound:Bennett}
\mathcal{I}_{v,+}^{\theta}(M)
\le 
\inf_{\alpha>0}\left\{
\frac{1}{\alpha} \log\left(
\frac{{b_v}^2}{{b_v}^2+\sigma_v^2}\exp(-\alpha\sigma_v^2/{b_v})+\frac{\sigma_v^2}{{b_v}^2+\sigma_v^2}\exp(\alpha{b_v})
\right)
+\frac{M}{\alpha}\right\} 
 \, .
\end{equation}

\begin{figure}[tbh]
\begin{center}
\includegraphics[width=0.8\textwidth]{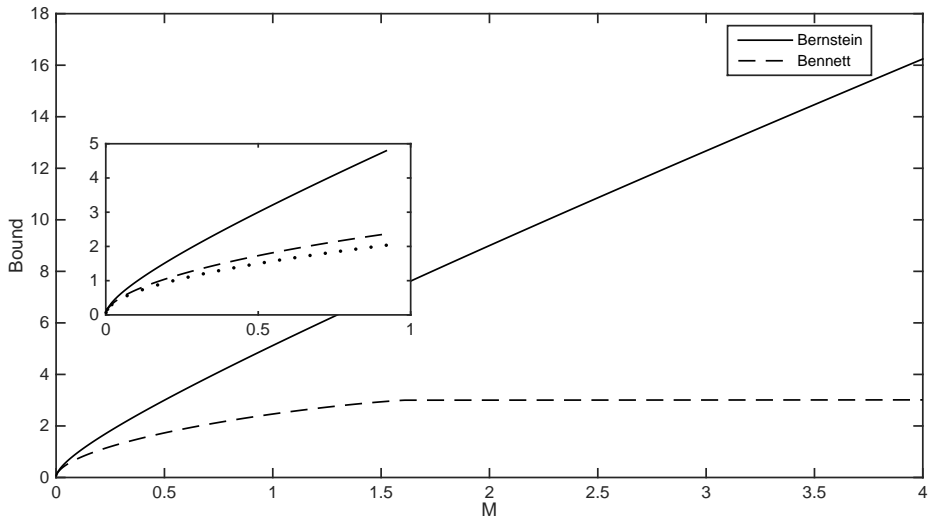}
\end{center}
\caption{ Bernstein (solid) \eqref{eq:conc:bound:Bernstein} and Bennett (dashed) 
\eqref{eq:conc:bound:Bennett} concentration inequality bounds for fixed values of $b_v$ and $\sigma_v$. The Bennett inequality is sharper and in fact it is also tight at $M=\infty$ where its asymptotic limit is $b_v\ge \mathbb{E}_{P^{\theta}_{\pm\alpha_\pm}}  \left[ v^{T}W^{\theta} \right]$. The inset plot zooms in near the origin where the event is not so rare and the bounds have a behaviour close to the linearized bound (dots) \eqref{SB:bound:linear}.}
\label{ci:bounds:fig}
\end{figure}


%
\end{remark}

\begin{remark}
[Linearized bounds]
If an  event $A$ is not rare, i.e. $P^{\theta}(A)\approx1$ or $-\log
P^{\theta}(A)\approx 0$, the calculation of the sensitivity indices $\mathcal{I}_{v,\pm}^{\theta}(M)$ is fairly trivial. Indeed, we can expand the sensitivity index 
\eqref{SB:bound:expl}
in a power series
in $\sqrt{-\log P^{\theta}(A)}$, if we 
assume that 
\begin{equation}
\label{eq: linearization:smallness}
M=-\log P^{\theta}(A)\ll1\, .
\end{equation}
Then the sensitivity index in the bound \eqref{sens:bound:gen:0}
is approximated as
\begin{equation}
\mathcal{I}_{v,\pm}^{\theta}(M)=
\pm \sqrt{2v^{T}\mathcal{F}%
(P^{\theta})vM}+O(M)\ ,
\label{SB:bound:linear}%
\end{equation}
where  $\mathcal{F}(P^{\theta})$ is again the Fisher information matrix for the parametric family 
$P^{\theta}$. The expansion  
\eqref{SB:bound:linear}
follows from the observation that $\mathbb{E}_{P^{\theta}}[\nabla_{\theta}\log p^{\theta}]=0$ and a direct application of Proposition~\ref{opt:bound:linear:gen:prop} together with the formula
for $(H_v^\theta)''(0)=v^T  \mathcal{F}(P^{\theta}) v$.
We also note that the Fisher information matrix arises when sensitivity
bounds for ordinary observables are linearized \cite[Theorem 2.14]{Dupuis:16}.
However, in order to get the bound in (\ref{SB:bound:linear}) two
linearizations were performed here; one capturing the closeness between
$P^{\theta+\epsilon v}$ and $P^{\theta}$ and another capturing that the event
is actually not rare. 
Interestingly, the linearization \eqref{SB:bound:linear} is also identical to the dominant term in the Bernstein bound 
\eqref{eq:conc:bound:Bernstein} for small $M=-\log P^{\theta}(A)$, i.e.  when  
$$M<\frac{2v^T  \mathcal{F}(P^{\theta}) v}{b_v}\, .$$


\end{remark}

\begin{remark} 
Although the bounds \eqref{eq:conc:bound:Bernstein} and \eqref{eq:conc:bound:Bennett} are  less tight than the index 
$\mathcal{I}_{v,+}^{\theta}(M)$, 
they are much easier to estimate than $\mathcal{I}_{v,+}^{\theta}(M)$ since they only involve the sampling of the Fisher information matrix $\mathcal{F}(P^{\theta})$, whose  calculation does not entail rare event sampling. Finally,  the 
sensitivity bounds  
\eqref{eq:conc:bound:Bernstein} and \eqref{eq:conc:bound:Bennett}
  hold for any $M$ and they are not just asymptotically true  in $M$, unlike the linearization 
 \eqref{SB:bound:linear} that  requires  \eqref{eq: linearization:smallness},  and in addition has an uncontrolled higher order error term $O(M)$. 
 %
\end{remark}



}

{
\medskip
\noindent
\textbf{Direct statistical estimation methods  for $\mathcal{I}_{v,\pm}^{\theta}(M)$.}
The two  representations of the sensitivity indices \eqref{SB:bound:expl}
in Theorem~\ref{thm:index},  either as a variational problem, or using the Kullback-Leibler divergence  
suggest at least two approaches to estimate the indices $\mathcal{I}_{v,\pm}^{\theta}(M)$ using  direct numerical simulation. 

First, since the optimization problem in the variational representation in \eqref{SB:bound:expl} is one dimensional, it is fairly trivial to solve numerically, hence  the main roadblock is the estimation of the  cumulant generating function $H_{v}^{\theta}(\pm \alpha)$.  Existing numerical methods to tackle either one of these problems already exist in the literature.  For instance, for the calculation of moment and cumulant generation functions can be performed using interacting particle systems methods \cite{DM:04,DM:05} or splitting techniques \cite{Lec:09,altalt1,gar2,deadup2}.

%
%
%

Using the alternative representation \eqref{SB:bound:expl}  of the minimizer $\mathcal{I}_{v,\pm}^{\theta}(M)$,  namely,
\begin{equation}\label{sec5:thermo}
\mathcal{I}_{v,\pm}^{\theta}(M)=\mathbb{E}_{P^{\theta}_{\pm\alpha_\pm}}  \left[ v^{T}W^{\theta} \right]\, . 
\end{equation}
%
demonstrates  the need to sample from the tilted measures  $P^{\theta}_{\pm\alpha_\pm}$, which  in turn is also intimately  related problem to estimating the cumulant generating function $H_{v}^{\theta}(\pm \alpha)$.
When $\pm \alpha_{\pm}$ is fairly close to zero, the sampling distribution $P^{\theta}_{\pm\alpha_\pm}$ in \eqref{sec5:thermo} is a perturbation of $P^\theta$;  in this  case the  Free Energy Perturbation method, see \cite{Lelievre:10} Section 2.4.1, can be used to simulate  efficiently $\mathcal{I}_{v,\pm}^{\theta}(M)$. However, by 
Proposition~\ref{aux:bound:func:gen:prop},  $\alpha_{\pm}(M)$ is increasing  in $M$ and thus when  $M=-\log P^{\theta}(A)$ grows, $\alpha_{\pm}$ in \eqref{sec5:thermo} can be large, see also the examples in Figure 2.  In this case $P^{\theta}_{\pm\alpha_\pm}$ in \eqref{sec5:thermo} is not necessarily  a perturbation of $P^\theta$. Multilevel Monte Carlo techniques such as chaining methods, see Section 11.6 in \cite{bishop}, or Thermodynamic Integration, as in Section 3.1 in \cite{Lelievre:10}, or the RESTART method \cite{deadup2} could in principle be used in the calculation of the indices \eqref{SB:bound:expl}. 

These issues are outside of the scope of this paper and  we plan to return to this topic and related implementations of the sensitivity indices  in a follow-up publication.
%
%
%

}

%

\section{Examples}
\label{sec:examples}

\noindent
{\bf Exponential family of distributions.}
It is instructive to consider the
example of exponential families, that is,  the family of measures with densities
$p^{\theta}(x)$ given by
\begin{equation*}
p^{\theta}(x) = e^{\theta^{T} t(x) - F(\theta)} \, .\,
\end{equation*}
where $t(x)$ is the vector of sufficient statistics and $F(\theta) = \log\mathbb{E}%
_{R}[ e^{ \theta^{T} t}]$. The score function is then given by
\begin{equation*}
W^{\theta}(x)= \nabla_{\theta}\log p^{\theta}(x) = t(x) - \nabla F(\theta) \,=\,
t(x) -\mathbb{E}_{P^{\theta}}[t]\,,
\end{equation*}
and the cumulant generating function is given by
\[
\begin{aligned} H^\theta_v(\alpha) :&= \log\mathbb E_{P^\8}[\exp\{ \alpha v^T\nabla\log p^\8 \}] \\ &= \log \int \exp\left\{\alpha v^T(t(x) - \nabla F(\theta))\right\} \exp\{\theta^T t(x) - F(\theta)\} R(dx) \\ &= \log \int \exp\left\{(\alpha v+\theta)^Tt(x) - (F(\theta)+\alpha v^T \nabla F(\theta))\right\} R(dx)\\ &= F(\alpha v+\theta) - F(\theta) - \alpha v^T \mathbb{E}_{P^\theta}[t]
= F(\alpha v+\theta) - F(\theta) - \alpha v^T  \nabla F(\theta)\,. \end{aligned}
\]
It is worth noting that the cumulant generating function is the Bregman divergence,
\cite{Erven:14},
 associated with $F$ at points $\alpha v+\theta$ and $\theta$.
This is an explicit quantification of the cost to be paid for tilting the distribution in order to make the event less rare.
Additionally, the tilted measure $P^{\theta}_{\alpha}$ has density
\[
\frac{dP^{\theta}_{\alpha}}{dR}= \exp\{ \alpha v^{T}( t(x) -\mathbb{E}%
_{P^{\theta}} [t] )- H^{\theta}_{v} (\alpha) + \theta^{T} t(x) - F(\theta) \}
= \exp\{ (\theta+ \alpha v)^{T} t(x) - F(\theta+ \alpha v) \}=p^{\theta+
\alpha v} 
\]
and thus $P^{\theta}_{\alpha}= P^{\theta+ \alpha v}$ also belongs to the same
exponential family. Finally the optimal $\alpha_{\pm}$ are the solutions of
the equation
\[
\log P(A) = -\mathcal{R}\left( { P^{\theta+ \alpha v}}{\,||\,}{P^{\theta}%
}\right)  = F( \theta+ \alpha v) - F(\theta) + \alpha v^T \nabla F( \theta+ \alpha
v)\,,
\]
and the sensitivity bounds and corresponding indices $\mathcal{I}_{v,\pm}^{\theta}(M)$ in Theorem~\ref{thm:index}
take the form
\begin{equation}\label{sens:bounds:exp}
\mathcal{I}_{v,-}^{\theta}(M)= \mathbb{E}_{P^{\theta-\alpha_{-}v}}\left[ v^T t\right]  - \mathbb{E}%
_{P^{\theta}}\left[ v^T t \right]  \,\le\, S_{v}^{\theta}(A) \,\le\,\\
 \mathbb{E}_{P^{\theta+\alpha_{+}v}}\left[ v^T t \right]  - \mathbb{E}%
_{P^{\theta}}\left[ v^T t \right]=\mathcal{I}_{v,+}^{\theta}(M) 
\, , 
\end{equation}
i.e., they  are  expressed as the difference between the mean sufficient statistics
under the optimally tilted distributions and  the mean sufficient statistics
under the original distribution. 

\textrm{Finally, many exponential families have explicit formulas for the
cumulant generating function of the sufficient statistic. Thus, the sensitivity
bounds as well as the two characterizations of the optimal values $\alpha_\pm$
can be  visualized. Figures~\ref{gaussian:fig}--\ref{laplacian:fig} show
the upper bound function, $(H(\alpha)+M)/\a$, with ${H}(\alpha):=H_v^\8(\a)$
as well as the derivative of the cumulant generating function for various
distributions and various values of $M$. The minimizer of the upper bound,
$\alpha_+$, can be geometrically characterized as the intersection of the upper bound function
 and ${H}^{\prime}(\alpha)$. Similar plots are obtained for the lower bound.  
}

\begin{figure}[th]
\textrm{\centering
\subfigure[Gaussian distribution with $\8$ being the mean value over the variance. The density of Gaussian distribution is given by $p^\8(x) = e^{\8 x - F(\8)}$ where $F(\8)=\frac{1}{2}\sigma^2\8^2$ with $\sigma^2$ being the variance.]{		\includegraphics[width=.40\textwidth]{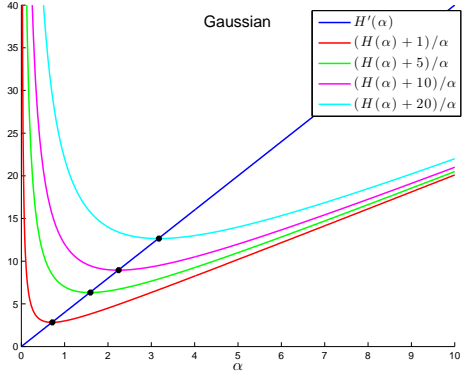}
		\label{gaussian:fig}}  \quad
\subfigure[Poisson distribution with $\8$ being the logarithm of the Poisson rate. The density of Poisson distribution is given by $p^\8(x) = e^{\8 x - F(\8)}$ where $F(\8)=e^\8$.]{
		\includegraphics[width=.40\textwidth]{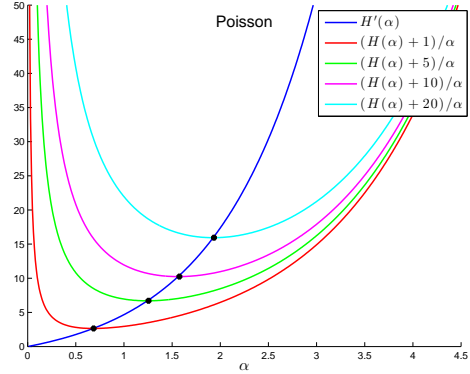}
		\label{poisson:fig}}  \quad}
\par
\textrm{\subfigure[Bernoulli distribution with $\8$ being the logarithm of the ratio $\frac{p}{1-p}$ where $p$ is the probability of success. The density of  Bernoulli distribution is given by $p^\8(x) = e^{\8 x - F(\8)}$ where $F(\8)=\log(1+e^\8)$.]{		\includegraphics[width=.40\textwidth]{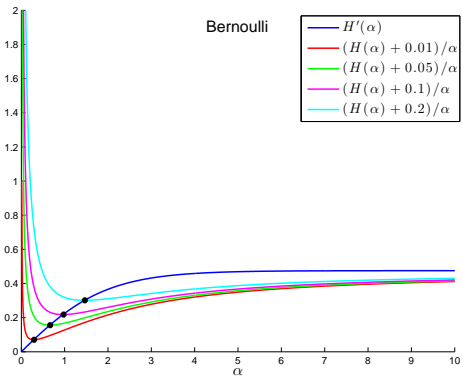}
		\label{bernoulli:fig}}  \quad
\subfigure[Centered Laplacian distribution with $\8$ being the negative of the mean value. The density of centered Laplacian distribution is given by $p^\8(x) = e^{\8 |x| - F(\8)}$ where $F(\8)=\log(-\frac{1}{\8})$.]{
		\includegraphics[width=.40\textwidth]{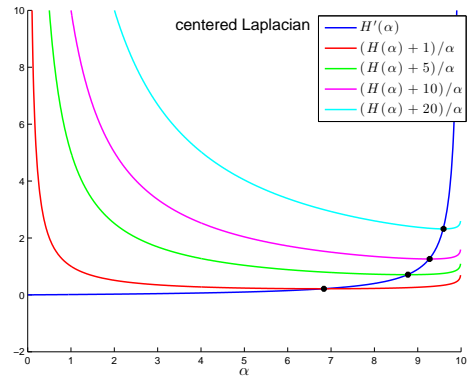}
		\label{laplacian:fig}}  \quad}
\caption{Graphical representation of the upper bound function, $(H(\a)+M)/\a$, and the derivative of cumulant generating function, $H^{\prime}(\alpha)$, for different members of the exponential family of distributions and various values of $M$. We assume the positive direction (i.e., $v=1$) thus $H(\a):=H_v^\8(\a) = F(\8+\a)-F(\8)-\a F^\prime(\8)$. The minimizer of the upper bound, $\alpha_{+}=\a_+(M)$, is depicted as a black dot. Despite having a unique minimum,  the shape of the bound function varies significantly across the distributions.}%
\label{exp:fam:fig}%
\end{figure}

\medskip
\noindent
{\bf Normal distribution.}
We next consider the specific case of a normal distribution and demonstrate that  the sensitivity bounds are  easier to implement
than the sampling of the rare event and the corresponding likelihood ratio method for rare events, discussed in Section 1. In this case we can consider the 
 tail events such like $A=\{ X > L\}$ or any other rare event $A$ characterized by the parameter $M$,
$$
M\ge -\log P^\theta(A)\, , \quad \mbox{where}\quad P^\theta= {\mathcal N}(\mu, \sigma^2)\, , 
$$
as well as the corresponding sensitivity index \eqref{LRR}:
\begin{equation}
\label{LRR:normal}
S_{v}^{\theta}(A)
= \mathbb{E}_{P^{\theta}_{|A}} v^T W^{\theta}
=\mathbb{E}%
_{P^{\theta}_{|A}}[v^T t(x)] -\mathbb{E}_{P^{\theta}}[v^T t]\, .
\end{equation}
On the other hand, we consider the bounds \eqref{sens:bounds:exp}, where the optimal values of $\pm\alpha_\pm$ are given by 
the explicitely solvable equation
\begin{equation}\label{alpha:normal}
\mathcal{R}\left( { P^{\theta}_{\pm\alpha_{\pm}}}{\,||\,}{P^{\theta}}\right)
=M \, .
\end{equation}
In fact, using the  notation of the exponential family we have that
$$
P^\theta=e^{\theta^{T} t(x) - F(\theta)}={\mathcal N}(\mu, \sigma^2)\, , \quad \mbox{where} \quad t(x)=(x, x^2/2)\, , \quad \theta=(\theta_1, \theta_2)=(\mu/\sigma^2, -\sigma^{-2})\, .
$$
Furthermore, for any unit vector $v=(v_1, v_2)$, 
$$
W^\theta(x)=(x-\mu, x^2/2- (\sigma^2+\mu^2)/2)\, , \mbox{ and}\quad  v^T W^{\theta}(x)=v_1(x-\mu)+v_2\left(x^2/2- (\sigma^2+\mu^2)/2)\right)\, .
$$
Next we define the corresponding tilted distributions 
\begin{equation}
\label{tilted:normal}
P^\theta_\alpha=P^{\theta+\alpha v}= {\mathcal N}\left(\frac{\mu+\alpha\sigma^2v_1}{1-\alpha\sigma^2v_2}, \frac{\sigma^2}{1-\alpha\sigma^2v_2}\right)\, .
\end{equation}
Using the normality of the distributions $P^\theta_\alpha$ and $P^\theta$ we have:
$$
\mathcal{R}\left( { P^{\theta}_{\alpha}}{\,||\,}{P^{\theta}}\right)=\frac{1}{2}\left[\frac{1}{1-\alpha\sigma^2v_2}-1+\log(1-\alpha\sigma^2v_2)+\alpha^2\sigma^2
\left(\frac{v_1+\mu v_2}{1-\alpha\sigma^2v_2}\right)
\right]\, .
$$
For instance,  we can consider the solution of \eqref{alpha:normal} in the case $v=(1, 0)$. Then we obtain that 
$\alpha_\pm=\sqrt{2M+1}/\sigma$ and by replacing it in \eqref{tilted:normal} we have 
the tilted distributions for upper and lower bounds in \eqref{sens:bounds:exp}
$$
P^{\theta}_{\pm\alpha_{\pm}}={\mathcal N}\left(\mu\pm\sigma\sqrt{2M+1}\, , \, \sigma^2\right)\, .
$$
Furthermore the bounds and the corresponding sensitivity indices  in \eqref{sens:bounds:exp}--where 
$S_{v}^{\theta}(A)$ is given by \eqref{LRR:normal}--take the following simple form: 
\begin{equation}
\label{eq:gaussian:bound}
\mathcal{I}_{v,\pm}^{\theta}(M)= 
\pm\sigma\sqrt{2M+1}\, .
\end{equation}
It is worth noting that the upper and lower bounds $\mathcal{I}_{v,\pm}^{\theta}(M)$
do not involve the sampling of each specific rare event $A$ as required by \eqref{LRR:normal}.

{
Finally, this example provides a demonstration of a parameter insensitivity analysis based on the upper and lower bounds $\mathcal{I}_{v,\pm}^{\theta}(M)$ in Theorem~\ref{thm:index}. Specifically, if $\sigma\sqrt{2M+1}\ll 1$
in \eqref{eq:gaussian:bound}, then we readily obtain from Theorem~\ref{thm:index} that the probability of the rare event $A$ is insensitive with respect to  the parameter $\theta_1=\mu/\sigma^2$. This insensitivity is quantified by the indices in \eqref{eq:gaussian:bound} without having to calculate explicitly the gradient  of $\log P^\theta(A)$ in the direction $v=(1, 0)$.
}

\section{Sensitivity Bounds for Large Deviation Rate Functions}
\label{sec:Sens:LDP}

Rare events are closely related to the theory of large deviations, \cite{Dupuis:97,Dembo:98}, and in this section 
we show how our sensitivity bounds and the sensitivity indices $\mathcal{I}_{v,\pm}^{\theta}(M)$ in
Theorem~\ref{thm:index} are related 
to large deviation rate functions.
We first present a  general framework  which applies to any large deviation principle, and subsequently  we discuss 
 more concrete problems, such as independent, identically distributed (IID)  and Markov sequences.
 
\medskip 
\noindent
{\bf General result.}
Recall that a sequence of probability measures $\{P_n\}$ on a Polish space $\mathcal{X}$ 
satisfies a large deviation principle if there exists a lower-semicontinuous function
$I: \mathcal{X} \to \mathbb{R}$ with compact level sets such that for any Borel set $A \subset \mathcal{X}$ 
\begin{equation}
- \underline{I} (A) \le  \liminf_{n\rightarrow\infty} \frac{1}{n}\log P_n( A) \leq 
 \limsup_{n\rightarrow\infty}\frac{1}{n}\log P_n( A) \leq- \overline I(A)\,, \label{LDP:1}%
\end{equation}
where, with $A^{\circ}$ denoting the interior of $A$ and $\bar{A}$ the closure of $A$, 
\begin{equation*}
\underline{I} (A) =  \inf_{x\in A^{\circ}}I(x)\,, \quad  \overline I(A) = \inf_{x\in\bar{A}}I(x) \,.
\end{equation*}
For ``nice sets'', e.g.  open and convex,  we have that  $\underline{I} (A)= \overline{I} (A)$ in 
which case the set $A$ is called an $I$-continuity set, \cite{Dembo:98}, although we do not need such an assumption  here.   

To perform the sensitivity analysis introduced in previous sections, we will consider a parametric family $P^\theta_n$ 
of probability measures satisfying a large deviation principle with rate function $I^\theta$; we also 
 fix an event of interest denoted by $A$.  { We intend to obtain  bounds on the relative rate of change  of the probability of the rare event $A$, which in view of the large deviation principle \eqref{LDP:1} translates to  bounds on the derivatives of the  functions ${\overline I^\theta}(A)$ and ${\underline  I^\theta}(A)$ with respect to $\theta$.}
We assume that $P^\theta_n$ has a density $p^\theta_n$ with respect to some reference 
measure $R(dx)$; we  consider the score function $W^\theta_n = \nabla_\theta \log p^\theta_n(x)$, as well as the sensitivity indices
\begin{equation}
S_{v,n}^{\theta}(A)=v^{T}\nabla_{\theta}\log P_{n}^{\theta}(A) \ ,
\label{SI:1n}%
\end{equation}
which we assume both to be well-defined.  Under the additional assumptions in Section \ref{sec:Sens:bounds}, see Theorem~\ref{sec5:thm1} and Theorem~\ref{thm:index}, we obtain the  following sensitivity bounds
\begin{equation} 
\label{eq:sensbounds:ldp:finiten}
\mathcal{I}_{v,n,-}^{\theta}(M)\le - \inf_{\alpha >0} \left\{ \frac{ H^\theta_{v,n}(-\alpha)  - \log P^\theta_n (A) }{\alpha} \right\}
\,\le\,    S_{v,n}^{\theta}(A) 
  \,\le\,
  \inf_{\alpha >0} \left\{ \frac{ H^\theta_{v,n}(\alpha)  - \log P^\theta_n (A)}{\alpha}
  \right\}\le 
  \mathcal{I}_{v,n,+}^{\theta}(M) \,,
\end{equation}
for all events $A$ such that  $M\ge -\log P^\theta_n (A)$, where 
\begin{equation*}{H}^\theta_{v,n}(\alpha)=\,\log\mathbb{E}_{{P_{n}^{\theta}}}[e^{\alpha 
v^{T} W^\theta_n} ]\,.
\end{equation*}
All quantities are indexed by $n$ to denote their dependence on the sequence of the probability measures $\{P_n\}$.
Using  \eqref{LDP:1} and \eqref{eq:sensbounds:ldp:finiten}
we obtain immediately the following result. 

\begin{theorem}[Sensitivity indices and large deviation limits] \label{ldp:sens:bound:gen} Assume that the limit 
\begin{equation*}
h^\theta_v(\alpha) \,=\, \lim_{n \to \infty} \frac{1}{n} {H}^\theta_{v,n}(\alpha)
\end{equation*}
exists and define the rare event sensitivity indices
\begin{equation}
\bar{s}_{v}^{\theta}(A):=\limsup_{n\rightarrow\infty} \frac{1}{n} S_{v,n}^{\theta}(A)\,,\quad \underbar s_{v}^{\theta}(A):=\liminf_{n\rightarrow\infty}%
\frac{1}{n} S_{v,n}^{\theta}(A)\,. \label{SI:LDP:upper}%
\end{equation}
Then we have 
\begin{equation*}
- \inf_{\alpha>0}  \left\{ \frac{h^\theta_v(- \alpha) + \bar{I}^\theta (A) }{\alpha} \right\}      
\leq\underbar s_{v}^{\theta
}(A)\leq\bar{s}_{v}^{\theta}(A)\leq \inf_{\alpha>0}  \left\{ \frac{h^\theta_v(\alpha) + \underbar{I}^{\theta} (A) }{\alpha} \right\}   .
\label{SB:LDP:Th1}%
\end{equation*}
Furthermore, as  in Theorem~\ref{thm:index},
we can define the sensitivity indices in terms of { $M$-level sets of the large deviation functionals}
\eqref{LDP:1}:
\begin{equation}\label{eq:SI:LDP:general}
\mathcal{I}_{v,\infty, -}^{\theta}(M):=- \inf_{\alpha>0}  \left\{ \frac{h^\theta_v(-  \alpha) + M}{\alpha} \right\}
\, , \quad 
\mathcal{I}_{v,\infty, +}^{\theta}(M):= \inf_{\alpha>0}  \left\{ \frac{h^\theta_v(\alpha) + M}{\alpha} \right\} \,.
\end{equation}

\end{theorem}

In order to obtain more precise and concrete results and representations of the sensitivity indices we consider next  some standard examples from the theory of large deviations.

\smallskip
\noindent
{\bf Sequences of IID random variables.} For IID sequences one can, unsurprisingly, 
bound the sensitivity of the large deviation rate function
in terms of the moment generating function of the score function.

\begin{theorem}[Sensitivity indices for IID random sequences] Let $X$ be a random vector taking values in $\mathbb{R}^d$ with probability distribution 
$P^\theta$ and density  $p^\theta = \frac{dP^\theta}{dR}$ with respect to some reference measure 
$R$. Assume that the score function $W^\theta=\nabla_\theta \log p^\theta$  
satisfies the integrability  conditions of Theorem \ref{sec5:thm1} with cumulant generating function
$H^\theta_v(\alpha) = \log  \mathbb{E}_{P^\theta}[ e^{\alpha v^T W^\theta}]$.
Assume that the moment generating function $\mathbb{E}_{P^\theta} [e^{ \lambda^T X}]$  
is finite for $\lambda \in \mathbb{R}^d$ in the neighbourhood of the origin  
so that, if  $X_i$, $i=1, 2, \cdots $ is  a sequence of IID 
random variables with common distribution $P^\theta$, then $S_n = \frac{1}{n} \sum_{k=1}^n X_i$ satisfies a large  deviation  principle with rate function 
\[ 
I^\theta(x) = \sup_{\lambda \in \mathbb{R}^d} \left\{ \lambda^T x -   \log \mathbb{E}_{P^\theta} [ e^{ \lambda^T X} ] \right\} \,.
\] 
The sensitivity indices defined in  \eqref{SI:LDP:upper} then satisfy  
\begin{equation}
- \inf_{\alpha>0} \left\{ \frac{ H^\theta_v(-\alpha)  + \bar{I}^\theta (A) }{\alpha} \right\}      
\leq\underbar s_{v}^{\theta
}(A)\leq\bar{s}_{v}^{\theta}(A)\leq \inf_{\alpha>0}  \left\{ \frac{H^\theta_v(\alpha)  + \underbar{I}^{\theta} (A) }{\alpha} \right\}   
\label{SB:LDP:IID1} \, .
\end{equation} 
Moreover, if $\underbar{I}^\theta (A)<M_+$ ($\bar{I}^\theta (A)<M_-$), where $M_\pm$ are as in Proposition~\ref{aux:bound:func:gen:prop}, then there exist finite $\alpha_{\pm}$ such that 
\begin{equation}
\label{SB:bound:cond1}
\mathbb{E}_{P^\theta_{-\alpha_-}}[ v^T W^\theta] 
  \leq\underbar s_{v}^{\theta}(A)\leq\bar{s}_{v}^{\theta}(A)\leq \mathbb{E}_{P^\theta_{\alpha_+}}[ v^T W^\theta] 
\end{equation}
with 
\begin{equation}
\mathcal{R}\left(  P^{\theta}_{-\alpha_{-}} {\,||\,} P^{\theta}  \right)
= \bar{I}^\theta(A)\,, \quad  \mathcal{R}\left(  P^{\theta}_{\alpha_{+}}{\,||\,}{P^{\theta}}\right)
= \underbar{I}^\theta(A) \,,\label{SB:bound:cond2}%
\end{equation}
and $P^\theta_\alpha$ is the measure with density $\frac{dP^\theta_\alpha}{dP^\theta} = e^{ \alpha v^T W^\theta- H^\theta_v(\alpha)}$. Finally we can define sensitivity indices $\mathcal{I}_{v,\infty, \pm}^{\theta}$
similarly to \eqref{eq:SI:LDP:general}.
\end{theorem}

\begin{proof}  Let $Q_n$ denote the distribution of $\frac{S_n}{n}$ and  $P_{n}^{\theta}=P^{\theta}\times \cdots
\times P^{\theta}$ the joint distribution of  $(X_1,  \cdots, X_n)$  which has  density $p_n^\theta(x_1, 
\cdots, x_n) =\prod_{1=1}^n p^\theta(x_i)$ with respect to $R_n=R \times \cdots \times R$.  
We have then 
\[
\nabla_\theta \log Q_n(A)  \,=\, \nabla_\theta \log P_{n}^{\theta} \left\{ (x_1, \cdots, x_n)\,: \frac{1}{n}\sum_{i=1}^n x_i \in A \right\}. 
\]

The score function for the probability $P_{n}^{\theta}$ is given by  
\[
W^\theta_n(x_1, \cdots , x_n)  =   \nabla_\theta \log p_n^\theta(x_1, \cdots, x_n) = \sum_{k=1}^n 
W^ \theta( x_k) 
\]
and thus we obtain 
\[
\frac{1}{n} \log \mathbb{E}_{P_{n}^{\theta}}\left[  e^{ \alpha v^T W^\theta_n(X_1, \cdots ,X_n)}\right]  =
\frac{1}{n} \log \mathbb{E}_{P_{n}^{\theta}}\left[ \prod_{k=1}^n e^{ \alpha v^T W^\theta(X_k)}\right] \,=\,
 \log \mathbb{E}_{P^{\theta}}\left[  e^{ \alpha v^T W^\theta(X)} \right] \,.
\]
Using Theorem \ref{sec5:thm1} and taking $n \to \infty$ we obtain \eqref{SB:LDP:IID1}.  Finally, using Proposition~\ref{aux:bound:func:gen:prop} we obtain the representation \eqref{SB:bound:cond1}
and \eqref{SB:bound:cond2}. 
\end{proof}

\bigskip
\noindent
{\bf Markov sequences.} We can apply our results to  
any stochastic processes  for which a large deviation 
principle holds, but here we concentrate on the simplest 
case of discrete-time Markov chains  (DTMC) with finite state 
space where the rate function is easy to obtain
and we can dispense with technical assumptions.  
Let $\{X_{k}\}$ be an irreducible finite-state Markov chain with state space 
$\Sigma$ and transition matrix  $\pi^\theta(i,j)$ depending on a parameter vector 
$\theta 
\in \mathbb{R}^k$. We assume that $\pi^\theta$ generates an ergodic Markov
chain, that $\pi^\theta(i,j)$ depends smoothly on $\theta$ and that for any 
$i \in \Sigma$ the transition probabilities 
$\pi^{\theta}(i,j)$ and $\pi^{\theta+  \epsilon v }(i,j)$ are mutually 
absolutely continuous for $\epsilon$ in a neighborhood of $0$ and for all 
$v \in \mathbb{R}^k$ with $\|v\|=1$. We then define the score function 
\[
W^\theta(i,j) = \nabla_\theta \log \pi^{\theta}(i,j)\,,
\]
if $\pi^{\theta}(i,j) >0$, and set it equal to $0$ otherwise.
We assume, for simplicity, that the Markov chain starts in the (arbitrary) 
state $x_0$. The joint probability distribution of the Markov chain
on the time interval from $0$ to $n$ is 
\[
{P}_{n}^{\theta}(x_{1},...,x_{n})=\pi^{\theta}(x_{0}%
,x_{1})\times\cdots\times\pi^{\theta}(x_{n-1},x_{n})\,.
\]

For any $f: \Sigma \to \mathbb{R}$, the sequence of random variables  
$S_{n}=\frac{1}{n}\sum_{k=1}^{n}f(X_{k})$ satisfies a large deviation 
principle with a rate function which can be identified in terms of relative 
entropy \cite[Theorem 8.4.3]{Dupuis:97} { and the contraction principle for large deviations \cite[Theorem 1.3.2]{Dupuis:97}. } Let $l(i),i\in\Sigma$ denote a 
probability measure on the state space $\Sigma$ and let 
$\bar{\pi}(i,j)$ denote a transition kernel on $\Sigma$. 
For $\beta\in\mathbb{R}$ define%
\begin{equation}
I^\theta(\beta)={\inf_{\bar{\pi}, l} }\left\{  \sum_{i\in\Sigma}R\left(  \bar{\pi}(i,\cdot)\left\Vert
\pi^\theta(i,\cdot)\right.  \right)  l(i):\sum_{i\in\Sigma}f(i)l(i)=\beta,\sum
_{j\in\Sigma}l(j)\bar{\pi}(j,i)=l(i)\text{ for }i\in\Sigma\right\}  .
\label{eq:rate:Markov}
\end{equation}
The second constraint in this definition  implies that $l$ is invariant under
$\bar{\pi}$, while  the first constraint  enforces  that the mean of $f$ under $l$ is
$\beta$. The rate function $I^\theta=I^\theta(\beta)$ is then the minimum of a relative entropy cost for
\textquotedblleft tilting\textquotedblright\ from $\pi$ to $\bar{\pi}$, so
that the mean of $f$ under the stationary distribution of $\bar{\pi}$ is equal to  $\beta
$. An alternative representation of $I^\theta$ is as the Legendre transform 
of the $\log \lambda^\theta(\alpha)$, where $\lambda^\theta(\alpha)$ is the maximal 
eigenvalue of the positive matrix $\pi^\theta(i,j)e^{\alpha f(j)}$, \cite{Dembo:98}.  

Now we can directly deduce from Theorem \ref{ldp:sens:bound:gen} the following result on rare event sensitivity bounds for Markov sequences.
\begin{theorem}[Sensitivity indices for Markov sequences] 
Let 
$\{X_k\}_{k=1,2,\cdots}$ be a ergodic Markov chain on the finite state space 
$\Sigma$ with transition probabilities $\pi^\theta(i,j)$. Assume that 
the transition probabilities $\pi^\theta$ depend smoothly on $\theta$ and 
that for all  $i \in \Sigma $,  $\pi^\theta(i,j)$ and 
$\pi^{\theta + \epsilon v}(i,j)$ are mutually absolutely continuous for 
$\epsilon$ sufficiently small and $\|v\|=1$. 
For $f: \Sigma \to \mathbb{R}$ the sequence $S_n= \frac{1}{n}\sum_{k=1}^n f(X_k)$ satisfies a large deviation principle with rate function $I^\theta(\beta)$ given in \eqref{eq:rate:Markov} and we have the sensitivity bounds 
\begin{equation}
- \inf_{\alpha>0} \left\{ \frac{ h^\theta_v(-\alpha)  + \bar{I}^\theta (A) }{\alpha} \right\}      
\leq\underbar s_{v}^{\theta
}(A)\leq\bar{s}_{v}^{\theta}(A)\leq \inf_{\alpha>0}  \left\{ \frac{h^\theta_v(\alpha)  + \underbar{I}^{\theta} (A) }{\alpha} \right\}   \ ,
\label{SB:LDP:Markov1}%
\end{equation}
where 
\begin{equation*}
h^\theta_v(\alpha)\,=\, \lim_{n\to \infty}\frac{1}{n}
\mathbb{E}_{P^{\theta}_n}\left[ e^{\alpha v^T \sum_{k=1}^nW^\theta(X_{k-1}, X_{k})}\right]\,.
\label{SB:LDP:Markov2}%
\end{equation*}
\end{theorem}

We can represent $h^\theta_v(\alpha)$ as the logarithm of the maximal eigenvalue 
of the matrix $p^\theta(i,j) e^{\alpha v^T W^\theta(i,j)}$. 
Alternatively, consider the empirical measure for pairs $(X_{i-1},X_{i})$, which has the rate function
\[
J(l\otimes\bar{\pi})=\sum_{i\in\Sigma}R\left(  \bar{\pi}(i,\cdot)\left\Vert
\pi^{\theta}(i,\cdot)\right.  \right)  l(i)
\]
whenever $l$ is stationary under $\bar{\pi}$ and $[l\otimes\bar{\pi
}](i,j)=l(i)\bar{\pi}(i,j)$ ($J$ is infinity otherwise). 
Using this rate function, we can evaluate $h^\theta_v(\alpha)$, since by Varadhan's lemma we have  
\begin{align*}
&  \lim_{n\rightarrow\infty}\frac{1}{n}\log\mathbb{E}_{P_{n}^{\theta}%
}e^{\alpha v^{T}\sum_{i=1}^{n} W^\theta(X_{i-1},X_{i}%
)}\\
&  \quad=\sup\left\{  \alpha v^{T}\sum_{i,j\in\Sigma}W^\theta(i,j)l(i)\bar{\pi}(i,j)-\sum_{i\in\Sigma}R\left(  \bar{\pi
}(i,\cdot)\left\Vert \pi(i,\cdot)\right.  \right)  l(i):\sum_{j\in\Sigma
}l(j)\bar{\pi}(j,i)=l(i)\text{ for }i\in\Sigma\right\}  .
\end{align*}
Similarly to
\eqref{eq:rate:Markov}, this last variational problem is  easily solved as a constrained convex programming problem.
We demonstrate such an implementation in the next example.

Indeed,  as a concrete example to illustrate the previous results we  consider a Markov chain with five states with values
$-2,...,2$. The transition probabilities are defined as
\[
\encircle{-2}  \underset{1-p_{-1}}{\overset{1}{\rightleftarrows}}
\encircle{-1} \underset{1-p_0}{\overset{p_{-1}}{\rightleftarrows}}
\encircle{0} \underset{1-p_1}{\overset{p_0}{\rightleftarrows}}
\encircle{1} \underset{1}{\overset{p_1}{\rightleftarrows}} \encircle{2} \ .
\]

Let $\theta=[p_{-1},p_{0},p_{1}]^{T}$ be the parameter vector and let
$\{X_{i}\}$ be a Markov chain created from the above transition laws, and let
$S_{n}$ be the empirical average of $X_{1},\ldots,X_{n}$. We study the
tightness of the sensitivity bound for the event $A=\{S_{n}=z\}$. Although $A$
is not an $I^{\theta}$-continuity set, using the special lattice structure of
the set of states in the support of  the empirical measure, we  can easily show that 
 $\lim_{n\rightarrow\infty}\frac{1}{n}\log P_{n}^{\theta}(S_{n}%
=z_{n})=-I^{\theta}(z)$, as long as $z_{n}\rightarrow z$ , when each
$z_{n}$ is of the form $i/n$. The rate functional $I^{\theta}(z)$ is finite in
the interval $[-1.5,1.5]$
and infinite elsewhere due to the fact that the transitions from state $2$
(resp. $-2$) cannot happen more than $n/2$ times making $\frac{1}{n}\sum
_{i=1}^{n}x_{i}\in\lbrack-1.5,1.5]$.

In Figure~\ref{si:sb:x:fig} we present the sensitivity indices and the associated
sensitivity bounds for various values of $z$ and various perturbation
vectors $v$. Similarly, in  Figure~\ref{si:sb:theta:fig} we depict the sensitivity
indices and the associated sensitivity bounds for various values of the
parameter vector. Both figures suggest that the sensitivity bounds are
informative and tightly follow the true  values of the gradient sensitivity indices \eqref{SI:LDP:upper}, which in the present setting coincide.

\begin{figure}[tbh]
\begin{center}
\includegraphics[width=\textwidth]{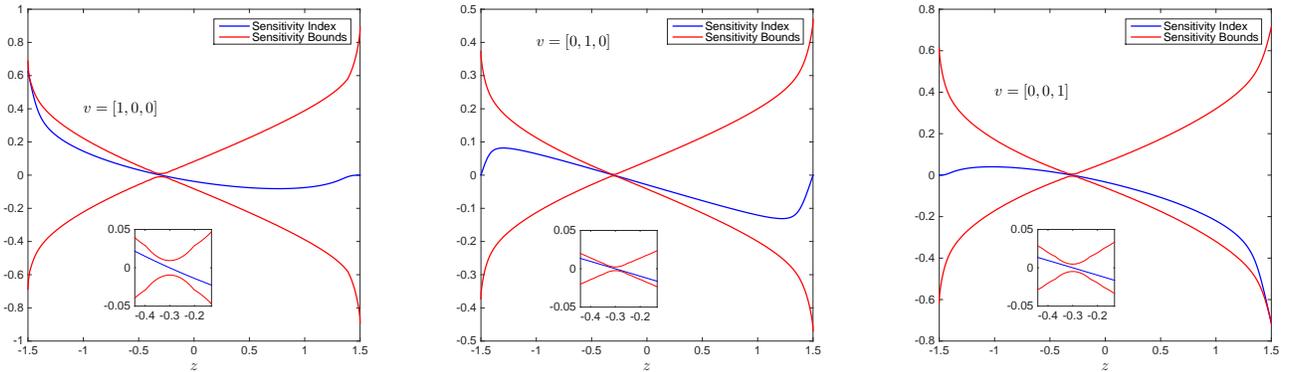}
\end{center}
\caption{ { Sensitivity index \eqref{SI:LDP:upper} (blue) } and the associated sensitivity bounds in { \eqref{SB:LDP:Markov1} (red)} for various
levels of rare event probability as quantified by the variation in the values of $z$.
Each panel corresponds to a different perturbation vector
$v$. From left to right the perturbation vectors are the three orthonormal
unit vectors. Parameter vector is kept fixed at $\theta=[0.2,0.5,0.7]^{T}$.
Inset plots zoom around the zero value for the sensitivity index.}%
\label{si:sb:x:fig}%
\end{figure}

\begin{figure}[tbh]
\begin{center}
\includegraphics[width=\textwidth]{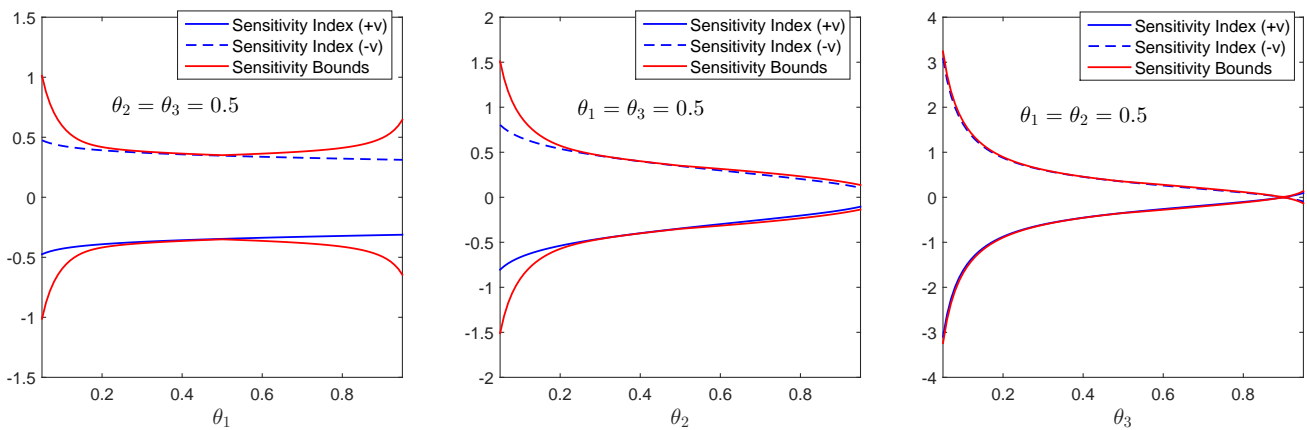}
\end{center}
\caption{Sensitivity index for $v$ and $-v$ and the associated sensitivity
bounds at different parameter regimes. Both the rare event and the perturbation vector
are kept fixed at $x=1$ and $v=\frac{1}{\sqrt{3}}[1,1,1]^{T}$. In each
panel, one of the parameters is varying while the other two are kept
constant. Evidently, sensitivity bounds tightly follow the sensitivity indices
in this demonstration.}%
\label{si:sb:theta:fig}%
\end{figure}

\appendix

\section{Appendix: Proposition and proofs} \label{app:a}

Let $g \in \mathcal{E}$ (see Definition \ref{def:gdef})
and denote by $g_{\pm} = \pm {\rm essup} \{ \pm 
g(x)\}$ its upper/lower bound (we allow the value $+ \infty$). Recall that $H(\alpha) = \log 
\mathbb{E}_P[e^{\alpha g}]$ is the moment generating function with 
Legendre transform $L(\delta) =  \sup_{\alpha \in \mathbb{R}}\left\{ 
\alpha \delta - H(\alpha)\right\}$ and that $P_\alpha$ is the 
exponential family  with $\frac{dP_\alpha}{dP} = e^{\alpha g - 
H(\alpha)}$.

\begin{proposition}
\label{CGF:LT:gen:prop} For $g \in \mathcal{E}$ we have 

\begin{itemize}
\item[(a)] The map $H(\alpha)$ is a convex function of $\alpha$, is 
finite in an interval $(d_-,d_+)$ with $d_-< 0 < d_+$ and 
$H(\alpha)=\infty$ for $\alpha \notin [d_-,d_+]$.  

In the interval $(d_-,d_+)$ 
the map $H(\alpha)$ is infinitely differentiable and strictly 
convex unless $g$ is constant $P$-a.s.;  
we have $H'(\alpha) = \mathbb{E}_{P_\alpha}[g]$ and 
$H''(\alpha)= \mathrm{Var}_{P_{\alpha}}(g)$. 

\item[(b)] The map $L(\delta)$ is a convex, non-negative and lower semi-continuous function of $\delta$ and $L(\mathbb{E}_{P} [g]) =0$.  If $g$ is not constant $P$-a.s. then $L(\delta)$ is strictly convex in the interval 
$(H'(d_-),  H(d_+))$ and for any $\delta \in (H'(d_-), H(d_+))$ there exists $\alpha \in (d_-,d_+)$ such that
$H'(\alpha)=\delta$
and 
\[
L(\delta)=\alpha H'(\alpha) -H(\alpha) = \mathcal{R}\left({P_{\alpha}}{\,||\,}{P}\right)  \ .
\]

\end{itemize}
\end{proposition}

\begin{proof} These are standard results used in the theory of large deviations, see  e.g. \cite[Lemma~2.2.5, Exercise~2.2.24]{Dembo:98}.
\end{proof}

We turn next to the proof of Proposition \ref{aux:bound:func:gen:prop}. Most ingredients 
in the proof have appeared in various recent papers by (some of) the authors and their collaborators \cite{Chowdhary:13,Dupuis:16,Gourgoulias:17}.
The formulation here is slightly different and since the results play a central role in the paper 
we provide a proof for convenience and completeness.

\begin{proof}[Proof of Proposition \ref{aux:bound:func:gen:prop}]
First note that it is enough to prove the result for $H(\alpha)$ since the 
result for $H(-\alpha)$ is obtained by replacing $g$ by $-g$. 

We first claim that automatically%
\[
H(d_{+})=\lim_{\alpha\uparrow d_{+}}H(\alpha),
\]
where $H(d_{+})$ may be infinite. By monotone convergence
\[
\mathbb{E}_{P}[1_{\{g\geq0\}}e^{\alpha g}]\uparrow\mathbb{E}_{P}%
[1_{\{g\geq0\}}e^{d_{+}g}]
\]
as $\alpha\uparrow d_{+}$. By dominated convergence%
\[
\mathbb{E}_{P}[1_{\{g<0\}}e^{\alpha g}]\downarrow\mathbb{E}_{P}[1_{\{g<0\}}%
e^{d_{+}g}]
\]%
as $\alpha\uparrow d_{+}$, and the claim follows. A very similar argument shows that $H^{\prime
}(\alpha)$ also has a limit as $\alpha\uparrow d_{+}$.

Let
\begin{equation}\label{eq:B:Appendix}
B(\alpha;M)=\frac{H(\alpha)+M}{\alpha}.
\end{equation}
We divide into cases. Always we have $H(0)=0$.

\begin{enumerate}
\item $g_{+}<\infty$. In this case $H^{\prime}(\alpha)\uparrow g_{+}<\infty$
as $\alpha\rightarrow\infty$ and $H^{\prime}(0)<g_{+}$. If $M=0$ then the
infimum is $H^{\prime}(0)$ and attained at $\alpha_{+}=0$. If $M>0$ then let
\[
B^{\prime}(\alpha;M)=\frac{\alpha H^{\prime}(\alpha)-H(\alpha)-M}{\alpha^{2}}%
\]
for $\alpha\geq0$. The function $\alpha H^{\prime}(\alpha)-H(\alpha)$ strictly
increases from $0$ at $\alpha=0$ to some limit $M_{+}>0$ at $\alpha=\infty$,
and the minimizer is at the unique finite root of $\alpha H^{\prime}%
(\alpha)-H(\alpha)=M$ for $M<M_{+}$ and $\alpha_{+}=\infty$ for $M\geq M_{+}$.

\item $g_{+}=\infty$. In this case there are two subcases.

\begin{enumerate}
\item $d_{+}=\infty$. In this case since $g_{+}=\infty$ we have $H^{\prime
}(\alpha)\uparrow\infty$ as $\alpha\rightarrow\infty$ and $\alpha H^{\prime
}(\alpha)-H(\alpha)\rightarrow\infty$ as $\alpha\rightarrow\infty$. Since  
$0H^{\prime}(0)-H(0)=0$, in all cases of $M\geq0$ there is a unique root
to $\alpha H^{\prime}(\alpha)-H(\alpha)=M$ and hence a unique minimizer.

\item $d_{+}<\infty$. We know that $H^{\prime}(\alpha)$ converges as
$\alpha\uparrow d_{+}$ to a well defined left hand limit which we call
$H^{\prime}(d_{+})$ (note that this value could be $\infty$). Thus we have
that $\alpha H^{\prime}(\alpha)-H(\alpha)$ ranges from $0$ at $\alpha=0$ to
$M_{+}=d_{+}H^{\prime}(d_{+})-H(d_{+})$. For $M\in\lbrack0,M_{+})$ there is a
unique minimizer in $[0,d_{+})$. For $M\geq M_{+}$ the unique minimizer is at
$\alpha_{+}=d_{+}$.
\end{enumerate}
\end{enumerate}

 To conclude the proof we note that if 
 $\alpha_+ < d_+$ then  
 $$\alpha_+ H'(\alpha_+)- H(\alpha_+) = \mathcal{R}(P_{\alpha_+}\,||\,P)=M
\,, $$
 and thus 
 \[
 B(\alpha_+, M) = H'(\alpha_+)=\mathbb{E}_{P_{\alpha_+}[g]}
 \]
which proves \eqref{eq:Bpmeqn} and \eqref{eq:cond:re}.  Finally if $d_+=\infty$ and $g$ is $P$-a.s. bounded above then the infimum is equal to $\lim_{\alpha \to \infty}\frac{H(\alpha)}{\alpha}$ and this establishes \eqref{eq:gplus}.
If $d_+ < \infty$ and $M_+ < \infty$ then the bound takes the form \eqref{eq:dplus}.

\end{proof}

\bibliographystyle{unsrt}
\bibliography{ref-sensitivity-rare-events}

\end{document}